\documentclass[11pt]{amsart}

\usepackage{latexsym, amssymb, amsthm,color}
\usepackage[centertags]{amsmath}
\usepackage{pictexwd}
\usepackage[normalem]{ulem}
\usepackage[all]{xy}

\newcommand{\dnote}[2][red]{{\leavevmode\unskip\footnotesize\hspace{0.75em}\textcolor{#1}{[#2]}}}
\newcommand{\todo}[1]{\dnote{\textbf{todo:} #1}}

\definecolor{GREEN}{rgb}{0,1,0}
\definecolor{green4}{rgb}{.1,.5,.1}
\definecolor{blue}{rgb}{0,0,1}
\definecolor{gray}{rgb}{0.5,0.5,0.5}

 \newcommand{\ep}{\end{proof}}

 \newif\ifpctex

  \newtheorem{theorem}{Theorem}
  \newtheorem{definition}{Definition}[section]
  \newtheorem{defi}[definition]{Definition}

  \newtheorem{cond}[definition]{Condition}
  \newtheorem{proposition}[definition]{Proposition}
  \newtheorem{lemma}[definition]{Lemma}

  \newtheorem{corollary}[definition]{Corollary}
  \newcommand{\beCond}[2]{\Rand{\vspace{0,6cm}\tt #1}\begin{cond}[#2]
  \label{#1}} \theoremstyle{definition}
  \newtheorem{remark}[definition]{Remark}
  
  \numberwithin{equation}{section}
  \newtheoremstyle{step}{3pt}{0pt}{\itshape}{}{\bf}{}{.5em}{}

\theoremstyle{step} \newtheorem{step}{Step}

\newcommand{\E}{\mathbb{E}}
\newcommand{\M}{\mathbb{M}}
\newcommand{\PP}{\mathbb{P}}
\newcommand{\R}{\mathbb{R}}

\newcommand{\N}{\mathbb{N}}

\newcommand{\T}{\mathbb{T}}

\newcommand{\CB}{\mathcal{B}}
\newcommand{\CC}{\mathcal{C}}
\newcommand{\CD}{\mathcal{D}}

\newcommand{\CP}{\mathcal{P}}

\newcommand{\CS}{\mathcal{S}}

\newcommand{\FC}{\mathfrak{C}}
\newcommand{\Fs}{\mathfrak{s}}
\newcommand{\Ft}{\mathfrak{t}}
\newcommand{\Fc}{\mathfrak{c}}

\newcommand{\Rand}[1]{\marginpar{#1}} 
\newcommand{\be}[1]{\begin{equation}\label{#1}}
\newcommand{\ee}{\end{equation}}
\newcommand{\bew}[1]{\Rand{\vspace{0,6cm}\tt #1}\begin{equation*}\label{#1}}
\newcommand{\eew}{\end{equation*}}
\newcommand{\bea}[1]{\Rand{\vspace{0,6cm}\tt #1}\begin{eqnarray*}\label{#1}}
\newcommand{\eea}[1]{\end{eqnarray*}}

\newcommand{\beL}[2]{\Rand{\vspace{0,6cm}\tt #1}\begin{lemma}[#2]\label{#1}}
\newcommand{\beD}[2]{\Rand{\vspace{0,6cm}\tt #1}\begin{definition}[#2]\label{#1}}
\newcommand{\beT}[2]{\Rand{\vspace{0,6cm}\tt #1}\begin{theorem}[#2]\label{#1}}
\newcommand{\beP}[2]{\Rand{\vspace{0,6cm}\tt #1}\begin{proposition}[#2]\label{#1}}
\newcommand{\beC}[1]{\Rand{\vspace{0,6cm}\tt #1}\begin{corollary}\label{#1}}
\newcommand{\beR}[1]{\Rand{\vspace{0,6cm}\tt #1}\begin{remark}[#1]\label{#1}}

\DeclareMathAlphabet{\mathpzc}{OT1}{pzc}{m}{it}

\begin{document}

\title[The algebraic $\alpha$-Ford tree under evolution]{The algebraic $\alpha$-Ford tree under evolution}

\author{Josu\'e Nussbaumer}
\address{Josu\'e Nussbaumer\\ Fakult\"at f\"ur Mathematik\\
Universit\"at Duisburg-Essen, Campus Essen\\  Universit\"atsstra{\ss}e~2\\
 45132 Essen \\ Germany}
 \email{josue.nussbaumer@uni-due.de}

\author{Anita Winter}
\address{Anita Winter \\ Fakult\"at f\"ur Mathematik\\
Universit\"at Duisburg-Essen, Campus Essen\\  Universit\"atsstra{\ss}e~2\\
 45132 Essen \\ Germany}
\email{anita.winter@uni-due.de}

\thanks{Research was supported by the DFG through the SPP Priority Programme 1590.}

\thispagestyle{empty}


\subjclass[]{}

  \begin{abstract}
    Null models of binary phylogenetic trees are useful for testing hypotheses on real world phylogenies. In this paper we consider
    phylogenies as binary trees without edge lengths together with a sampling measure and encode them as
    algebraic measure trees. This allows to describe the degree of similarity between actual and simulated phylogenies by focusing  on the sample shape of subtrees and their subtree masses. We describe the annealed law of the statistics of subtree masses of null models, namely the branching tree, the  coalescent tree, and the comb tree in more detail. Finally, we use methods from martingale problems to characterize evolving phylogenetic trees in the diffusion limit.
    \end{abstract}

 \maketitle

 {
 \tableofcontents
 }

 \section{Introduction and motivation}
 \label{S:introduction}
 {An {\em $N$-cladogram} is a semi-labeled, un-rooted and binary
tree with $N\ge 2$
{\em leaves} labeled $\{1,2,...,N\}$ and with $N-2$ unlabeled internal nodes.
Cladograms are particular phylogenetic trees for which no information on the edge lengths is available, and which therefore only capture the tree structure.

As prototype models are needed for testing real world phylogenies, parametric families of random cladograms have been studied (compare~\cite{Aldous1996,Ford2005}). One such family introduced in \cite{Ford2005} is today referred to as the {\em $\alpha$-Ford model} {(see also \cite{HaasMiermontPitmanWinkel2008,ChenFordWinkel2009,PitmanWinkel2009,Stefansson2009,CoronadoMirRossello2018})}.
Fix $\alpha\in[0,1)$ and $N\in\mathbb{N}$. The {\em $\alpha$-Ford tree} with $N$ leaves is an $N$-cladogram
constructed recursively as follows (compare Figure~\ref{Fig03}):
\begin{enumerate}
        \item Start with one edge, and label its leaves by $\{1,2\}$ (yielding the only $2$-cladogram).
        \item Given the $\alpha$-Ford tree with {$k\ge 2$} leaves, assign weight $1-\alpha$ to each external and weight $\alpha$ to each internal edge.
        \item Choose an edge at random according to these weights and to the middle of this edge, insert a new leaf together with an edge. Label the new leaf $k+1$.
        \item Stop when the current  binary combinatorial tree has $N$ leaves.
        \item Permute the leaf labels.
        \end{enumerate}
        Note that permuting the labels in the last step ensures {\em consistency}. That is, for all $1\le m\le N$, restricting to the sub-cladogram spanned by a uniform sample of size $m$ from the {leaf set} $\{1,...,N\}$ yields an $m$-cladogram which equals in law the $\alpha$-Ford tree with $m$ leaves.

The case $\alpha=1$ is excluded as for $k=2,3$ all edges have weight $0$, and therefore the above construction is not well-defined. However,
we can extend the construction with some care.
To overcome the issue, let us  simply choose the edge, at which we are inserting the next edge, uniformly among the external edges. As soon as $k=4$, there is only one possible tree shape with exactly one inner edge and the problem disappears.

The $\alpha$-Ford model interpolates continuously between three popular models ranging from the
coalescent tree (also known as Yule tree) in the case  $\alpha=0$  via the branching tree (also known as uniform tree) in the case $\alpha=\frac{1}{2}$
to the
totally unbalanced tree (also known as comb tree) in the case $\alpha=1$. In this paper we are interested in limit cladograms as the number of leaves goes to infinity.
For that we will rely on the notion of continuum
algebraic measure trees recently introduced in \cite{LoehrWinter}.

In what follows, we refer to $(T,c)$ as an {\em algebraic tree} if $T\not=\emptyset$ is a
set equipped with a {\em branch point map} $c\colon T^3\to T$ satisfying consistency conditions (see Definition~\ref{Def:001}).
Even though algebraic trees can be seen as metric trees where one has ``forgotten'' the metric, the branch point map is defined such that the notion of leaves, branch
points, degree, subtrees, line segments, open sets, etc.\ can be formalized without reference to a metric and
agree with the corresponding notion in the metric tree. An algebraic measure tree $(T,c,\mu)$ consists of a
separable algebraic tree $(T,c)$ together with a probability measure $\mu$ on the Borel $\sigma$-algebra ${\mathcal B}(T)$.
The $\alpha$-Ford diffusion limit takes values in the state space
\begin{equation}
\label{e:T2}
   \mathbb{T}_2:=\big\{(T,c,\mu)\in\mathbb{T}:\,\mbox{degrees at most $3$, atoms of $\mu$ only at leaves}\big\}
\end{equation}
of (equivalence classes of) binary algebraic measure trees with no atoms on the skeleton, and more specifically in its subspace
\begin{equation}
\label{e:Tcont}
   \mathbb{T}^{\mathrm{cont}}_2:=\big\{(T,c,\mu)\in\mathbb{T}_2:\,\mbox{$\mu$ non-atomic}\big\}
\end{equation}
of so called {\em continuum} binary algebraic measure trees.
We equip $\mathbb{T}_2$ with the  so-called 
{\em sample shape convergence} (Definition~\ref{Def:002}), which says that a sequence
$(\mathfrak{t}_N)_{N\in\mathbb{N}}$ converges to $\mathfrak{t}$ in $\mathbb{T}_2$ if the random shapes $\mathfrak{s}_{(T,c)}(x_1,...,x_m)$ of sub-cladograms spanned by finite samples $(x_1,...,x_m)$ of size $m$ converge weakly with respect to the discrete topology
(compare Definition~\ref{Def:003} and Figure~\ref{f:shapefunction}).
It is shown in \cite{LoehrWinter} that both $\mathbb{T}_2$ and
$\mathbb{T}_2^{\mathrm{cont}}$ are compact, which is very convenient for showing tightness.

To get started we first introduce the $\alpha$-Ford models with an infinite number of leaves in $\mathbb{T}_2$. To do this, we consider the $\alpha$-Ford tree with $N$ leaves as a random element in the subspace
\begin{equation}
\label{e:003}
	\T_2^N:=\big\{(T,c,\mu)\in\T_2:\#\mathrm{lf}(T,c)=N\mbox{ and }\mu=\tfrac{1}{N}\sum\nolimits_{u\in\mathrm{lf}(T,c)}\delta_u\big\},
\end{equation}
where $\mathrm{lf}(T,c)$ denotes the set of leaves.
Then, using the consistency property of Ford models, we can show that, for each $\alpha\in[0,1]$, the sequence of such constructed random binary algebraic measure trees converges to an element of $\T_2^\mathrm{cont}$, that we call \emph{$\alpha$-Ford algebraic measure tree} (with infinite number of leaves). For $\alpha=\frac{1}{2}$, we get the algebraic measure Brownian CRT, which is the unique continuum random algebraic measure tree
whose i.d.d.\ samples span uniform binary trees. For $\alpha=0$, we call this tree the \emph{Kingman algebraic measure tree} as it equals in law the algebraic measure tree read off from the Kingman coalescent.


In statistical applications of phylogenies with edge lengths, it has been exploited that
all sufficient information about genealogies is contained in the lengths of subtrees
spanned by a finite sample. One such example is the Watterson estimator for the mutation rate of a neutral population, which counts the number of segregating site that is often represented by the edge lengths (\cite{Watterson1975,BaakeHaeseler1999}). In this paper we  want to introduce with the {\em sample subtree mass distribution} a similar statistics which is more suited for algebraic measure trees, for which a priori edge lengths are not defined.
For that, consider for a branch point $v\in \mathrm{br}(T)$} the three subtree
components attached to $v$ and denote for each $u\not =v$ by $\CS_v(u)$ the subtree component that contains $u\in T$ (see \eqref{e:components} below for a precise definition).
For $\underline{u}=(u_1,u_2,u_3) \in T^3$, let
\begin{equation}\label{e:eta3}
	\underline{\eta}(\underline{u}) := \big(\eta_i(\underline{u})\big)_{i=1,2,3} := \big( \mu(\CS_{c(\underline{u})}(u_i)) \big)_{i=1,2,3}
\end{equation}
be the vector of the three masses of the components connected to $c(\underline{u})$. We refer to its annealed law as {\em sample subtree mass distribution}. It allows to distinguish between $\alpha$-Ford models for different $\alpha\in[0,1]$. For $\alpha=1$ it can be easily read off from the associated  comb tree (\cite[Proposition~1.6.8]{Winter}). For $\alpha=\frac{1}{2}$ a more elaborate combinatorial argument shows that is equal to the Dirichlet distribution with all parameter $(\frac{1}{2},\frac{1}{2},\frac{1}{2})$ (compare~\cite[Theorem~2]{Aldous1994} or \cite[Proposition~5.2]{LoehrMytnikWinter}). The case $\alpha=0$ is treated in Proposition~\ref{P:001} where we show that the sample subtree mass distribution of the Kingman algebraic measure tree equals in distribution the symmetrization of {$(B_{1,2}B_{2,2},B_{1,2}(1-B_{2,2}),1-B_{1,2})$},
where {$B_{1,2}$ and $B_{2,2}$} are independent Beta distributed random variables with parameters {$(1,2)$ and $(2,2)$}, respectively.  \smallskip

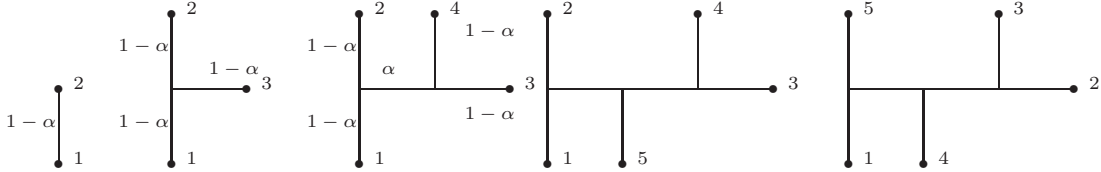
\begin{figure}
\setlength{\unitlength}{1cm}
\begin{picture}(30,3)(1,0)
%
\put(2.5,0){\line(0,1){1}}
\put(2.5,0){\circle*{.1}}
\put(2.7,0){{\tiny 1}}
\put(2.5,1){\circle*{.1}}
\put(2.7,1){{\tiny 2}}
\put(1.8,0.5){{\tiny $1-\alpha$}}
\put(4,0){\line(0,1){2}}
\put(4,0){\circle*{.1}}
\put(4.2,0){{\tiny 1}}
\put(4,2){\circle*{.1}}
\put(4.2,2){{\tiny 2}}
\put(4,1){\line(1,0){1}}
\put(5,1){\circle*{.1}}
\put(5.2,1){{\tiny 3}}
\put(3.3,0.5){{\tiny $1-\alpha$}}
\put(3.3,1.5){{\tiny $1-\alpha$}}
\put(4.5,1.2){{\tiny $1-\alpha$}}
\put(6.5,0){\line(0,1){2}}
\put(6.5,0){\circle*{.1}}
\put(6.7,0){{\tiny 1}}
\put(6.5,2){\circle*{.1}}
\put(6.7,2){{\tiny 2}}
\put(6.5,1){\line(1,0){2}}
\put(8.5,1){\circle*{.1}}
\put(8.7,1){{\tiny 3}}
\put(7.5,1){\line(0,1){1}}
\put(7.5,2){\circle*{.1}}
\put(7.7,2){{\tiny 4}}
\put(5.8,0.5){{\tiny $1-\alpha$}}
\put(5.8,1.5){{\tiny $1-\alpha$}}
\put(7.9,1.7){{\tiny $1-\alpha$}}
\put(7.9,.6){{\tiny $1-\alpha$}}
\put(6.8,1.2){{{\tiny $\alpha$}}}
\put(9,0){\line(0,1){2}}
\put(9,0){\circle*{.1}}
\put(9.2,0){{\tiny 1}}
\put(9,2){\circle*{.1}}
\put(9.2,2){{\tiny 2}}
\put(9,1){\line(1,0){3}}
\put(12,1){\circle*{.1}}
\put(12.2,1){{\tiny 3}}
\put(11,1){\line(0,1){1}}
\put(11,2){\circle*{.1}}
\put(11.2,2){{\tiny 4}}
\put(10,1){\line(0,-1){1}}
\put(10,0){\circle*{.1}}
\put(10.2,0){{\tiny 5}}
\put(13,0){\line(0,1){2}}
\put(13,0){\circle*{.1}}
\put(13.2,0){{\tiny 1}}
\put(13,2){\circle*{.1}}
\put(13.2,2){{\tiny 5}}
\put(13,1){\line(1,0){3}}
\put(16,1){\circle*{.1}}
\put(16.2,1){{\tiny 2}}
\put(15,1){\line(0,1){1}}
\put(15,2){\circle*{.1}}
\put(15.2,2){{\tiny 3}}
\put(14,1){\line(0,-1){1}}
\put(14,0){\circle*{.1}}
\put(14.2,0){{\tiny 4}}
\end{picture}
\caption{The $5$ steps in constructing the $\alpha$-Ford tree with $5$ leaves.}
\label{Fig03}
\end{figure}

{An important ingredient for several algorithms that
reconstruct cladograms from DNA data
are
Markov chains that} move through a space of finite trees
(see, for example, \cite{Felsenstein2003} for a survey on Markov chain
Monte Carlo algorithms in maximum likelihood tree reconstruction).
The present paper  has a focus on the one-parameter family of Markov chains
on the space ${\mathcal C}_m$ of all $m$-cladograms which are related to the $\alpha$-Ford model in the following way.
Fix $\alpha\in[0,1]$. Rather than adding new leaves, we keep the number of leaves constant by first removing a leaf picked uniformly at random and then inserting it into an edge chosen at random according to the $\alpha$-Ford weights.  More detailed, {for each
pair $(k,e)$} consisting of a leaf and an edge (other than the edge adjacent to $k$) at rate $1$, the
Markov chain jumps from its current state $\mathfrak{t}$ to $\mathfrak{t}^{(k,e)}$, where the latter is
obtained as follows (see Figures~\ref{Fig:01a} and \ref{Fig:01b}):
\begin{itemize}
\item
erase the unique edge (including the incident vertices) which connects $k$ to the sub-tree spanned by all leaves but $k$,
\item
split the remaining subtree at the edge $e$ into two pieces, and
\item
reintroduce the above edge (including $k$ and the branch point) at the split point.
\end{itemize}

We call this Markov chain the \emph{$\alpha$-Ford chain} on $m$-cladograms. One can easily check that the $\alpha$-Ford model is the stationary distribution, and that
the $\alpha$-Ford chain is symmetric if and only if $\alpha=\frac{1}{2}$. In the latter case the mixing and relaxation time has been studied in detail in \cite{Aldous2000,Schweinsberg2001}. This case is therefore often referred to as the {\em Aldous move} or the {\em Aldous chain on cladograms}.
To see why the $\alpha$-Ford chain is  not symmetric for general $\alpha\in[0,1]$, notice that inserting a leaf at an edge creates a {\em cherry leaf} if and only if the edge was external. Therefore the time reversed $\alpha$-Ford chain picks at rate $(1-\alpha)$ a pair consisting of a cherry leaf and an edge,
and at rate $\alpha$ a pair consisting of a non-cherry leaf and an edge at random, and inserts the picked leaf at the chosen edge. The discrepancy $\beta_\alpha^m(\mathfrak{t})$ between the total backward and forward rate at the current state $\mathfrak{t}$ is a potential which links the forward $X^{m,\alpha}$ and backward $\alpha$-Ford chain $Y^{m,\alpha}$ via a {\em Feynman-Kac duality}: for all $\mathfrak{s},\mathfrak{t}\in\mathfrak{C}_m$,
\begin{equation}
\label{e:001}
   \mathbb{P}_{\mathfrak{s}}\big(\big\{X^{m,\alpha}_t=\mathfrak{t}\big\}\big)
   =
   \mathbb{E}_{\mathfrak{t}}\Big[\mathbf{1}_{\mathfrak{s}}
   \big(Y^{m,\alpha}_t\big)\exp\big(\int^t_0\beta_\alpha^m(Y^{m,\alpha}_s)\mathrm{d}s\big)\Big]
\end{equation}
(compare Proposition~\ref{P:Feynman}).

\begin{figure}
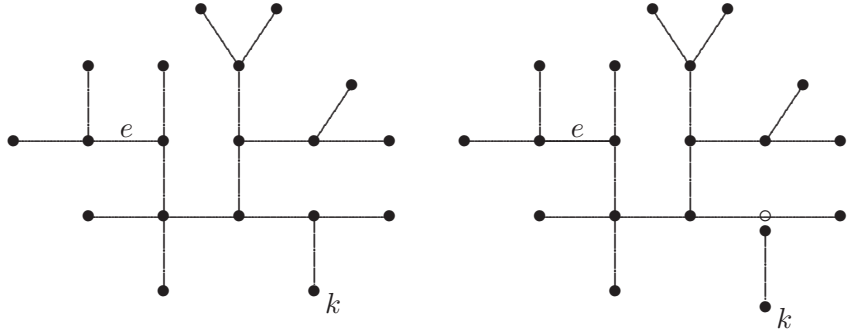

\centerline{
\beginpicture
	\setcoordinatesystem units <.5cm,.5cm>
	\setplotarea x from -0.5 to 22, y from -3.5 to 7
	\plot 2 0 10 0 /
	\plot 8 0 8 -1.78 /
	\plot 4 0 4 4 /
	\plot 2 2 0 2 /
	\plot 2 2 4 2 /
	\plot 2 2 2 4 /
	\plot 6 0 6 4 /
	\plot 6 2 10 2 /
	\plot 6 4 7 5.5 /
	\plot 8 2 9 3.5 /
	\plot 6 4 5 5.5 /
	\plot 4 0 4 -2 /
	\put{$\bullet$} [lC] at 3.8 -2
	\put{$\bullet$} [lC] at 4.8 5.5
	\put{$\bullet$} [lC] at 7.8 2
	\put{$\bullet$} [lC] at 9.8 2
	\put{$\bullet$} [lC] at 8.8 3.5
	\put{$\bullet$} [lC] at 6.8 5.5
	\put{$\bullet$} [lC] at 5.8 2
	\put{$\bullet$} [lC] at 5.8 4
	%
	\put{$\bullet$} [lC] at 1.8 4
	\put{$\bullet$} [lC] at 1.8 2
	\put{$\bullet$} [lC] at -.2 2
	\put{$\bullet$} [lC] at 3.8 2
	\put{$\bullet$} [lC] at 3.8 4
	\put{$\bullet$} [lC] at 1.8 0
	\put{$\bullet$} [lC] at 3.8 0
	\put{$\bullet$} [lC] at 5.8 0
	\put{$\bullet$} [lC] at 7.8 0
	\put{$\bullet$} [lC] at 7.8 -2
	\put{{$k$}} [IC] at 8.5 -2.3
	\put{$\bullet$} [lC] at 9.8 0
	\put{$e$} [IC] at 3 2.3
	\plot 14 0 22 0 /
	\plot 20 -.4 20 -2.18 /
	\plot 16 0 16 4 /
	\plot 16 2 14 2 /
	\plot 12 2 16 2 /
	\plot 14 2 14 4 /
	\plot 18 0 18 4 /
	\plot 18 2 22 2 /
	\plot 18 4 19 5.5 /
	\plot 20 2 21 3.5 /
	\plot 18 4 17 5.5 /
	\plot 16 0 16 -2 /
	\put{$\bullet$} [lC] at 15.8 -2
	\put{$\bullet$} [lC] at 16.8 5.5
	\put{$\bullet$} [lC] at 19.8 2
	\put{$\bullet$} [lC] at 21.8 2
	\put{$\bullet$} [lC] at 20.8 3.5
	\put{$\bullet$} [lC] at 18.8 5.5
	\put{$\bullet$} [lC] at 17.8 2
	\put{$\bullet$} [lC] at 17.8 4
	%
	\put{$\bullet$} [lC] at 13.8 4
	\put{$\bullet$} [lC] at 13.8 2
	\put{$\bullet$} [lC] at 11.8 2
	\put{$\bullet$} [lC] at 15.8 2
	\put{$\bullet$} [lC] at 15.8 4
	\put{$\bullet$} [lC] at 13.8 0
	\put{$\bullet$} [lC] at 15.8 0
	\put{$\bullet$} [lC] at 17.8 0
	\put{$\circ$} [lC] at 19.8 0
	\put{$\bullet$} [lC] at 19.8 -0.4
	\put{$\bullet$} [lC] at 19.8 -2.4
	\put{$k$} [IC] at 20.5 -2.7
	\put{$\bullet$} [lC] at 21.8 0
	\put{$e$} [IC] at 15 2.3
\endpicture
}
\caption{At rate $(1-\alpha)m(m-1)$ and rate $\alpha (m-4)$ a) a leaf $k$ and an external respectively internal edge $e$ are picked at random, and b) the edge adjacent to $k$ is taken away (leaving behind a branch point of degree $2$).}
\label{Fig:01a}
\end{figure}

\begin{figure}
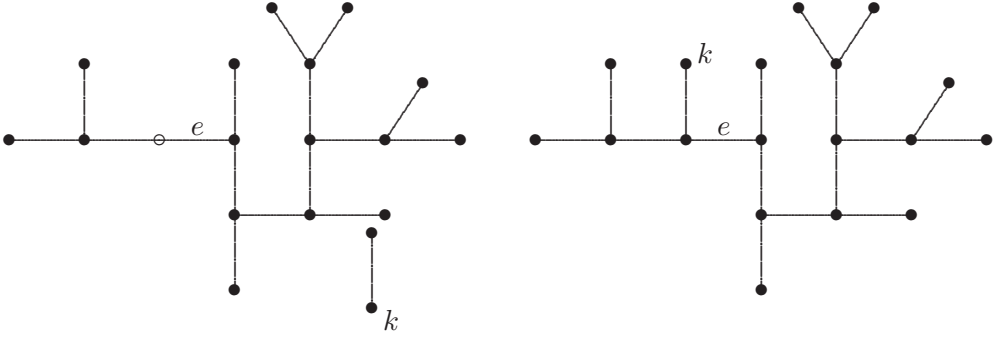

\centerline{
\beginpicture
	\setcoordinatesystem units <.5cm,.5cm>
	\setplotarea x from -2.5 to 12, y from -3.5 to 6
	\plot 4 0 8 0 /
	\plot 7.65 -2.5 7.65 -0.5 /
	\plot 4 0 4 4 /
	\plot 2 2 -2 2 /
	\plot 2 2 4 2 /
	\plot 0 2 0 4 /
	\plot 6 0 6 4 /
	\plot 6 2 10 2 /
	\plot 6 4 7 5.5 /
	\plot 8 2 9 3.5 /
	\plot 6 4 5 5.5 /
	\plot 4 0 4 -2 /
	\put{$\bullet$} [lC] at 3.8 -2
	\put{$\bullet$} [lC] at 4.8 5.5
	\put{$\bullet$} [lC] at 7.8 2
	\put{$\bullet$} [lC] at 9.8 2
	\put{$\bullet$} [lC] at 8.8 3.5
	\put{$\bullet$} [lC] at 6.8 5.5
	\put{$\bullet$} [lC] at 5.8 2
	\put{$\bullet$} [lC] at 5.8 4
	\put{$\bullet$} [lC] at -.2 4
	\put{$\circ$} [lC] at 1.8 2
	\put{$\bullet$} [lC] at -.2 2
	\put{$\bullet$} [lC] at -2.2 2
	\put{$\bullet$} [lC] at 3.8 2
	\put{$\bullet$} [lC] at 3.8 4
	\put{$\bullet$} [lC] at 3.8 0
	\put{$\bullet$} [lC] at 5.8 0
	\put{$\bullet$} [lC] at 7.8 0
	\put{$\bullet$} [lC] at 7.45 -0.5
	\put{$\bullet$} [lC] at 7.45 -2.5
	\put{$e$} [IC] at 3 2.3
	\put{{$k$}} [IC] at 8.15 -2.8
	\plot 18 0 22 0 /
	\plot 18 0 18 4 /
	\plot 16 2 12 2 /
	\plot 16 2 18 2 /
	\plot 14 2 14 4 /
	\plot 20 0 20 4 /
	\plot 20 2 24 2 /
	\plot 20 4 21 5.5 /
	\plot 22 2 23 3.5 /
	\plot 20 4 19 5.5 /
	\plot 18 0 18 -2 /
	\plot 16 2 16 4 /
	\put{$\bullet$} [lC] at 15.8 4
	\put{$k$} [IC] at 16.5 4.3
	\put{$\bullet$} [lC] at 17.8 -2
	\put{$\bullet$} [lC] at 18.8 5.5
	\put{$\bullet$} [lC] at 21.8 2
	\put{$\bullet$} [lC] at 23.8 2
	\put{$\bullet$} [lC] at 22.8 3.5
	\put{$\bullet$} [lC] at 20.8 5.5
	\put{$\bullet$} [lC] at 19.8 2
	\put{$\bullet$} [lC] at 19.8 4
	\put{$\bullet$} [lC] at 13.8 4
	\put{$\bullet$} [lC] at 15.8 2
	\put{$\bullet$} [lC] at 13.8 2
	\put{$\bullet$} [lC] at 11.8 2
	\put{$\bullet$} [lC] at 17.8 2
	\put{$\bullet$} [lC] at 17.8 4
	\put{$\bullet$} [lC] at 17.8 0
	\put{$\bullet$} [lC] at 19.8 0
	\put{$\bullet$} [lC] at 21.8 0
	\put{$e$} [IC] at 17 2.3
\endpicture
}
\caption{c) the two edges containing the branch point of degree $2$ are identified while the edge $e$ gets opened, and d) the free edge gets shuffled there and reattached.}
\label{Fig:01b}
\end{figure}

One of the main goals of this paper is to construct the diffusion limit of the $\alpha$-Ford tree as the number of leaves goes to infinity, and
to provide analytic characterizations. In the case $\alpha=\frac{1}{2}$ the existence of such a diffusion limit was conjectured by David Aldous in a seminar held at the Field Institute in 1999 and had been listed on his open problem list since. Only recently such a Aldous diffusion was constructed in two independent and different approaches
(\cite{FormanPalRizzoloWinkel2016, FormanPalRizzoloWinkel2018a,FormanPalRizzoloWinkel2018b,FormanPalRizzoloWinkel2018c} versus \cite{LoehrMytnikWinter}). We will here follow the approach of \cite{LoehrMytnikWinter} which relies on the notion of algebraic measure trees and the sample shape convergence,
and generalize their construction to all $\alpha\in[0,1]$.


Consider the operator $\Omega_\alpha$ acting on test functions of so-called \emph{sample shape polynomials}
\begin{equation}
\label{e:004}
   \Phi^{m,\Ft}(\chi):=\int\mu^{\otimes m}(\mathrm{d}\underline{u})\,{\bf 1}_\Ft(\Fs_{(T,c)}(\underline{u})),
\end{equation}
with $m\in\mathbb{N}$, $\chi=(T,c,\mu)\in\mathbb{T}_2$ and $\mathfrak{t}\in\mathfrak{C}_m$,
as follows:
\begin{equation}
	\Omega_\alpha\Phi^{m,\Ft}(\chi):=\int_T\mu^{\otimes m}(\mathrm{d}\underline{u})\widetilde{\Omega}_\alpha^m
{\bf 1}_\Ft(\Fs_{(T,c)}(\underline{u})),
\end{equation}
where ${\bf 1}_\Ft$ plays the role of the test function for $\widetilde{\Omega}_\alpha^m$, which denotes the generator of the $\alpha$-Ford Markov chain on the space of $m$-cladograms.

We state here our first main result. To do so, we identify as before an $N$-cladogram with an element of $\T_2^N$ by forgetting the leaf labels and adding the uniform distribution on the leaves. That is, in what follows the $\alpha$-Ford chain is a $\T_2^N$-valued Markov chain.

\begin{theorem}[The well-posed martingale problem]\label{t:martprob}
Let $\alpha\in[0,1]$ and $P_0$ be a probability measure on $\T_2^\mathrm{cont}$. For each $N\in\N$, let $X_0^N\in\T_2^N$ and assume that $X_0\rightarrow \chi$, where $\chi$ is distributed according to $P_0$. Then the $\alpha$-Ford chain $X^{N,\alpha}$ starting in $X_0^N$ converges weakly in Skorokhod path space w.r.t.\ the sample shape convergence to a $\T_2^\mathrm{cont}$-valued Feller process $X^\alpha$ with continuous paths.

Furthermore, $X^\alpha$ is the unique $\T_2^\mathrm{cont}$-valued Markov process $(X_t)_{t\geq0}$ such that $P_0$ is the distribution of $X_0$, and for all $\Phi\in\CD(\Omega_\alpha)$, the process $M:=(M_t)_{t\geq0}$ given by
\begin{equation}
	M_t:=\Phi(X_t)-\Phi(X_0)-\int_0^t\Omega_\alpha\Phi(X_s)\mathrm{d}s
\end{equation}
is a martingale.
\end{theorem}

We refer to the process from Theorem~\ref{t:martprob} as \emph{$\alpha$-Ford diffusion}, which is justified by the first part of the theorem.
We point out that the $\alpha$-Ford diffusion is dual to the backward \emph{$\alpha$-Ford chain}  through the following Feynman-Kac-duality relation: for all $m\in\N$ and $\Ft\in\FC_m$, the $\alpha$-Ford diffusion $X:=((T_t,c_t,\mu_t))_{t\geq 0}$ with initial law $P_0=\delta_{\chi}$, $\chi\in\mathbb{T}_2^{\mathrm{cont}}$, satisfies
\begin{equation}
\label{e:005}	
    \E_\chi^X\big[\Phi^{m,\Ft}(X_t)\big]
    =
    \E_\Ft^{Y^m}\Big[\Phi^{m,Y_t^m}(\chi)\exp\big(\int_0^t\beta^m_\alpha(Y_s^m)\mathrm{d}s\big)\Big],
\end{equation}
where 
$Y^m:=(Y_t^m)_{t\geq0}$ is the $\alpha$-Ford backward chain on $m$-cladograms started in $Y_0^m=\Ft$ (Proposition~5.3).

In order to provide representations of the sample subtree mass distribution for general $\alpha\in[0,1]$, we extend this martingale problem as follows.
We consider test functions of the following form, called \emph{mass polynomials of degree $3$}: for
$f\colon [0,1]^3 \to \R$ continuous,
\begin{equation}
\label{e:Phi}
   \Phi^{f}(T,c,\mu)
   :=
   \int f\big(\underline{\eta}(c(\underline{u}))\big) \,\mu^{\otimes 3}(\mathrm{d}\underline{u}),
\end{equation}
where $(T,c,\mu)\in \mathbb{T}_2$. One of the main results of \cite{LoehrWinter} is that $\Phi^f \in {\mathcal C}(\mathbb{T}_2)$.

For all $\alpha\in[0,1]$, we extend the domain of the operator $\Omega_\alpha$ to the set of mass polynomials $\Phi^f$ with $f$ twice continuously differentiable on $[0,1]$. We then put
\begin{equation}
\begin{aligned}
\label{e:Kin}
	\Omega_\alpha\Phi^f(\chi)=& \int\mu^{\otimes 3}(\mathrm{d}\underline{u})\left(\sum_{i,j=1}^{3}\eta_i(\delta_{ij}-\eta_j)\partial^2_{ij}f(\underline{\eta}(\underline{u}))+(2-\alpha)\sum_{i=1}^3(1-3\eta_i)\partial_i f(\underline{\eta}(\underline{u}))\right.\\
	& +(2-3\alpha)\sum_{i=1}^3\big(f(e_i)-f(\underline{\eta}(\underline{u}))\big)+\frac{\alpha}{2}\sum_{i\neq j=1}^3 \frac{{\bf 1}_{\eta_i\neq 0}}{\eta_i}\left(f\circ\theta_{i,j}(\underline{\eta}(\underline{u}))-f(\underline{\eta}(\underline{u}))\right)\\
	& +\left.\frac{\alpha}{2}\sum_{i\neq j=1}^3\left({\bf 1}_{\eta_j=0}-{\bf 1}_{\eta_i=0}\right)\partial_if(\underline{\eta}(\underline{u}))\right)
\end{aligned}
\end{equation}
where $\theta_{i,j}\colon\Delta_2\rightarrow\Delta_2$ denotes the migration operator on the two-simplex
\begin{equation}
	\Delta_2:=\{\underline{x}\in[0,1]^3:x_1+x_2+x_3=1\},
\end{equation}
which sends the vector $\underline{\eta}$ to the vector where we subtract $\eta_i$ from the $i$th entry (resulting in the entry zero) and add it to the $j$th entry (resulting in $\eta_i+\eta_j$), and $e_i=(\delta_{ij})_{i=1,2,3}$ is the $i$th unit vector.

Our second main result is the following:
\begin{theorem}[Extended martingale problem for subtree masses]
Let $\alpha\in[0,1]$ and $X=(X_t)_{t\geq0}$ be the $\alpha$-Ford diffusion on $\T_2^\mathrm{cont}$. Then for all mass polynomials $\Phi^f$ with $f\in\CC^3([0,1])$, the process $M^f:=(M_t^f)_{t\geq0}$ given by
\begin{equation}
	M_t^f:=\Phi^f(X_t)-\Phi^f(X_0)-\int_0^t\Omega_\alpha\Phi^f(X_s)\mathrm{d}s
\end{equation}
is a martingale.
\end{theorem}

Applying the operator $\Omega_\alpha$ to test functions $f^{\underline{k}}:\Delta_3\to[0,1]$, $\underline{k}=(k_1,k_2,k_3)\in\N^3$, of the form
\begin{equation}
	f^{(\underline{k})}(\underline{\eta})=\eta_1^{k_1}\eta_2^{k_2}\eta_3^{k_3},
\end{equation}
we obtain the following representation of the distribution of the subtree mass vector
$\underline{X}^\alpha_\infty$  of the $\alpha$-Ford algebraic measure tree. Obviously,  $\E[f^{(0,0,0)}(\underline{X}^\alpha_\infty)]=1$ and
\begin{equation}
\E[f^{(1,0,0)}(\underline{X}^\alpha_\infty)]=\E[f^{(0,1,0)}(\underline{X}^\alpha_\infty)]=\E[f^{(0,0,1)}(\underline{X}^\alpha_\infty)]=\frac{1}{3},
\end{equation}
for all $\alpha\in[0,1]$. Moreover, the following recursive relations hold:

\begin{corollary}[Moments of subtree mass distribution of $\alpha$-Ford] For all $\alpha\in[0,1]$ and $\underline{k}\in\mathbb{N}_0^3$,
\begin{equation}
\begin{aligned}
	\E\big[f^{(\underline{k})}(\underline{X}^\alpha_\infty)\big]
 &=
   \frac{1}{(S+3)(S+2-3\alpha)}\Big(\sum_{i=1}^3\mathbf{1}_{\{k_i\not =0\}}(k_i+1)(k_i-\alpha)
   \E\big[f^{(\underline{k}-e_i)}(\underline{X}^\alpha_\infty)\big]
   \\
	&\;\;+(2-3\alpha)\big({\bf 1}\{k_1=k_2=0\}+{\bf 1}\{k_2=k_3=0\}+{\bf 1}\{k_3=k_1=0\}\big)
  \\
	&\;\;\; +\frac{\alpha}{2}\sum_{i=1}^3{\bf 1}_{k_i=0}\sum_{j\neq i=1}^3
    \sum_{l_j=1}^{k_j}{k_j\choose l_j}\E\big[f^{(\underline{k}+(l_j-1)e_i-l_je_j)}(\underline{X}^\alpha_\infty)\big]\Big),
\end{aligned}
\end{equation}
where $S=k_1+k_2+k_3$.

Specifically, if $\alpha=0$, then for all $\underline{k}\in\mathbb{N}_0^3$,
\begin{equation}
\label{e:025}
\begin{aligned}
   \E\big[f^{\underline{k}}\big(\underline{X}_\infty^{0}\big)\big]
   = 4\frac{\prod_{j=1}^3\Gamma(k_j+2)}{\Gamma(S+3)}\sum_{1\leq i_1<i_2\leq 3}\frac{\Gamma(k_{i_1}+k_{i_2}+1)}{\Gamma(k_{i_1}+k_{i_2}+4)}.
\end{aligned}
\end{equation}
\label{Lem:001}
\end{corollary}\smallskip

\noindent{\bf Outline. }
The rest of the paper is organized as follows. In Section~\ref{S:framework} we introduce our state space of algebraic measure trees and recall its most important properties from \cite{LoehrWinter}. In Section~\ref{S:Static} we consider the static $\alpha$-Ford model and show that the algebraic measure tree obtained from the genealogy of a Kingman coalescent equals the $(\alpha=0)$-Ford model. In Section~\ref{S:evolving}
we then consider the $\alpha$-Ford chain on cladograms with a fixed number of leaves and state the Feynman-Kac duality relation to the
time reversed chain. In Section~\ref{S:diffusion limit} we construct the diffusion limit of the $\alpha$-Ford chain as the number of leaves goes to infinity as a solution of a well-posed martingale problem. In Section~\ref{S:applications} we extend this martingale problem to test functions which evaluate the sample subtree mass distribution and derive our recursive relations for the moments of the sample subtree mass distribution. We will get a more explicit representation for the sample subtree mass distribution in case of the Kingman algebraic tree.

\section{The state space: algebraic measure trees}
\label{S:framework}
In this section we introduce the state space. For that we rely on the framework of
algebraic measure trees, which was introduced in \cite{LoehrWinter}. All proofs can be found there.
In order to focus on the algebraic tree structure rather than the metric, the definition of a tree is based on axioms on the map which sends any three points to their branch point.

\begin{defi}[Algebraic tree]
An \emph{algebraic tree} is a non-empty set $T$ together with a symmetric map $c\colon T^3\rightarrow T$ satisfying the following:
\begin{enumerate}
	\item[(2pc)] For all $x_1,x_2\in T$, $c(x_1,x_2,x_2)=x_2$.
	\item[(3pc)] For all $x_1,x_2,x_3\in T$, $c(x_1,x_2,c(x_1,x_2,x_3))=c(x_1,x_2,x_3)$.
	\item[(4pc)] For all $x_1,x_2,x_3,x_4\in t$,
	\begin{equation}
		c(x_1,x_2,x_3)\in\{c(x_1,x_2,x_4),c(x_1,x_3,x_4),c(x_2,x_3,x_4)\}.
	\end{equation}
\end{enumerate}
We call $c$ the \emph{branch point map}. A \emph{tree isomorphism} between two algebraic trees $(T_i,c_i)$, $i=1,2$, is a bijective map $\phi\colon T_1\rightarrow T_2$ such that for all $x_1,x_2,x_3\in T_1$,
\begin{equation}
	\phi(c_1(x_1,x_2,x_3))=c_2(\phi(x_1),\phi(x_2),\phi(x_3)).
\end{equation}
\label{Def:001}
\end{defi}

For each point $x\in T$, we define an equivalence relation $\sim_x$ on $T\setminus\{x\}$ such that for all $y,z\in T\setminus\{x\}$, $y\sim_xz$ if and only if $c(x,y,z)\neq x$. For $y\in T\setminus\{x\}$, we denote by
\begin{equation}
\label{e:components}
	\CS_x(y):=\{z\in T\setminus\{x\}:z\sim_x y\}
\end{equation}
the equivalence class of $y$ for this equivalence relation $\sim_x$. We also call $\CS_x(y)$ the \emph{component} of $T\setminus\{x\}$ containing $y$. We introduce the following definitions to describe the tree structure of an algebraic tree $(T,c)$:
\begin{itemize}
	\item a \emph{subtree} of $T$ is a set $S\subseteq T$ such that $c(S^3)=S$,
	\item the \emph{degree} of $x\in T$ is the number of components of $T\setminus\{x\}$, and we write $\mathrm{deg}(x):=\#\{\CS_x(y):y\in T\setminus\{x\}\}$,
	\item a \emph{leaf} is a point $u\in T$ such that $\mathrm{deg}(u)=1$, and we write $\mathrm{lf}(T)$ for the set of leaves,
	\item a \emph{branch point} is a point $v\in T$ such that $\mathrm{deg}(v)\geq3$, or equivalently such that $v=c(x_1,x_2,x_3)$ for some $x_1,x_2,x_3\in T\setminus\{v\}$, and we denote by $\mathrm{br}(T)$ the set of branch points,
	\item for $x,y\in T$, we define the interval $[x,y]$ as
	\begin{equation}
		[x,y]:=\{z\in T:c(x,y,z)=z\},
	\end{equation}
	\item and we say that $\{x,y\}$ is an \emph{edge} if $x\neq y$ and $[x,y]=\{x,y\}$.
\end{itemize}

There is a natural Hausdorff topology on a given algebraic tree, namely the topology generated by the set of all components $\CS_x(y)$ with $x\neq y$, $x,y\in T$. We say that an algebraic tree $(T,c)$ is \emph{order separable} if it is separable w.r.t.\ this topology and has at most countably many edges. We further equip order separable algebraic trees with a probability measure on the Borel $\sigma$-algebra $\CB(T,c)$, which allows to sample leaves from the tree.

\begin{defi}[Algebraic measure trees]
A (separable) \emph{algebraic measure tree} $(T,c,\mu)$ is an order separable algebraic tree $(T,c)$ together with a probability measure $\mu$ on $\CB(T,c)$.

We say that two algebraic measure trees $(T_i,c_i,\mu_i)$, $i=1,2$ are \emph{equivalent} if there exist subtrees $S_i\subseteq T_i$ with $\mu_i(S_i)=1$, $i=1,2$ and a measure preserving tree isomorphism $\phi$ from $S_1$ onto $S_2$, i.e.\ $c_2(\phi(x),\phi(y),\phi(z))=\phi(c_1(x,y,z))$ for all $x,y,z\in S_1$, and $\mu_1\circ\phi^{-1}=\mu_2$. We define
\begin{equation}
	\T:=\text{set of equivalence classes of algebraic measure trees}.
\end{equation}
With an abuse of notation, we will write $\chi=(T,c,\mu)$ for the algberaic tree as well as the equivalence class.
\end{defi}

A first way to equip $\T$ with a topology is by associating an algebraic measure tree with a metric measure tree in $\M$ the space of \emph{metric measure spaces}, and define the convergence of algebraic measure trees in $\T$ as the Gromov-weak convergence (compare, for example, \cite{GrevenPfaffelhuberWinter2009}) of these associated metric measure trees. We first need to define the metric measure tree associated to an algebraic measure tree.

\begin{defi}[branch point distribution]
For an algebraic measure tree $\chi=(T,c,\mu)$, the \emph{branch point distribution} on $T$ is defined as
\begin{equation}
	\nu_{(T,c,\mu)}:=\mu^{\otimes 3}\circ c^{-1},
\end{equation}
and we associate $\chi$ with the metric measure tree $(T,r_\mu,\mu)\in\M$, where we put for $x,y\in T$,
\begin{equation}
	r_\mu(x,y):=\nu_\chi([x,y])-\frac{1}{2}\nu_\chi(\{x\})-\frac{1}{2}\nu_\chi(\{y\}).
\end{equation}
\end{defi}
The choice of the metric $r_\mu$ can be understood as follows: two points are close if the mass branching off the line segment connecting them is small rather than if the length of this line segment is small. We then say that a sequence of algebraic measure trees converges in the \emph{branch point distribution distance Gromov-weak topology} if the associated (through $r_\mu$) sequence of metric measure trees converges Gromov-weakly.

Because cladograms are by definition binary, it is enough for the purpose of the present paper to consider the
subspace of $\mathbb{T}$ consisting of binary trees. More precisely, we consider the subspace of binary algebraic measure trees with the property that the measure has atoms only (if at all) on the leaves on the tree.
\begin{equation}
	\T_2=\{(T,c,\mu)\in\T:\mathrm{deg}(v)\leq3~\forall v\in T, \mathrm{at}(\mu)\subseteq\mathrm{lf}(T)\},
\end{equation}
where we write $\mathrm{at}(\mu)$ for the set of atoms of $\mu$.
Under this extra condition, the notion of Gromov-weak convergence with respect to $r_\nu$ is equivalent to a more combinatorial notion of convergence.
In contrast to the Gromov-weak convergence which relies on sample distance matrices, this combinatorial notion make use of sample shapes. To introduce the latter, we first extend our previous definition of cladograms as follows.
\begin{defi}[$m$-cladogram]
For $m\in\N$, an \emph{$m$-labelled cladogram} is a binary, finite tree $C=(C,c)$ consisting only of leaves and branch points together with a surjective labelling map $\zeta:\{1,...,m\}\rightarrow\mathrm{lf}(C)$. An \emph{$m$-cladogram} $(C,c,\zeta)$ is an $m$-labelled cladogram such that $\zeta$ is also injective.

We call two $m$-labelled cladograms $(C_1,c_1,\zeta_1)$ and $(C_2,c_2,\zeta_2)$ \emph{isomorphic} if there exists a tree isomorphism $\phi$ from $(C_1,c_1)$ onto $(C_2,c_2)$ such that $\zeta_2=\phi\circ\zeta_1$. We then write
\begin{equation}
	\overline{\FC}_m:=\{\text{isomorphism classes of }m\text{-labelled cladograms}\}
\end{equation}
and
\begin{equation}
	\FC_m:=\{(C,c,\zeta)\in\overline{\FC}_m:\zeta\text{ injective}\}.
\end{equation}
\end{defi}

Note that an $m$-cladogram has exactly $m$ leaves (and $m-2$ branch points). An $m$-labelled cladogram can have less than $m$ leaves (and $m-2$ branch points) if a leaf has multiple labels.

We next define the shape function, which allows to associate $m$ ordered distinct leaves with a unique $m$-cladogram.
\begin{defi}[Shape function]
For a binary algebraic tree $(T,c)$, $m\in\N$, and $u_1,...,u_m\in T\setminus\mathrm{br}(T)$, there exists a unique (up to isomorphism) $m$-labelled cladogram
\begin{equation}
	\Fs_{(T,c)}(u_1,...,u_m)=(C,c_C,\zeta)
\end{equation}
with $\mathrm{lf}(C)=\{u_1,..,u_m\}$ and $\zeta(i)=u_i$, such that the identity on $\mathrm{lf}(C)$ extends to a tree homomorphism $\pi$ from $C$ onto $c(\{u_1,...,u_m\}^3)$, i.e.\ for all $i,j,k=1,...,m$,
\begin{equation}
	\pi(c_C(u_i,u_j,u_k))=c(u_i,u_j,u_k).
\end{equation}
We will refer to $\Fs_{(T,c)}(u_1,...,u_m)\in\overline{\FC}_m$ as the \emph{shape} of $u_1,...,u_m$ in $(T,c)$ (compare with Figure~\ref{f:shapefunction}).
\label{Def:003}
\end{defi}

\begin{figure}[t]
\[
\xymatrix@=1pc{
&u_1\ar@{-}[d]&&{\bullet}&&{\bullet}   &   &&&\\
{\bullet}\ar@{-}[r]&{\bullet}\ar@{-}[dr]&&&{\bullet}\ar@{-}[ul]\ar@{-}[ur]&   &&&&&   u_1\ar@{-}[dr]&&u_3\ar@{-}[d]&\\
&&{\bullet}\ar@{-}[r]&u_3\ar@{-}[r]&{\bullet}\ar@{-}[r]\ar@{-}[u]&u_4   &&&&&   &{\bullet}\ar@{-}[r]&{\bullet}\ar@{-}[r]&u_4\\
&u_2\ar@{-}[ur]&&&&   &&&&&   u_2\ar@{-}[ur]&&&}\]
\caption{A tree $T$ and the shape $\Fs_{(T,c)}(u_1,u_2,u_3,u_4)$. The cladogram is not isomorphic to the subtree $c(\{u_1,u_2,u_3,u_4\}^3)$ because $u_3\in [u_1,u_4]$.}
\label{f:shapefunction}
\end{figure}
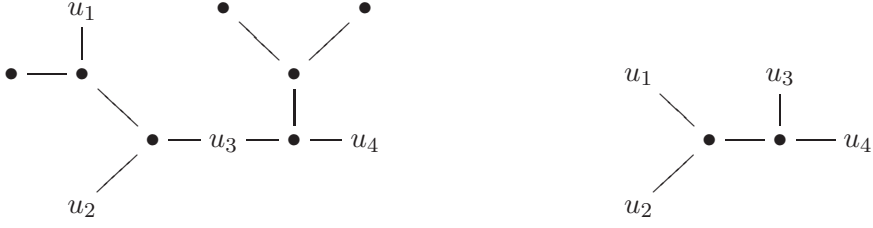

We are now in a position to define the sample shape convergence.
\begin{defi}[Sample shape convergence]
A sequence $(\chi_N)_{N\in\N}$ of binary algebraic measure trees $(T_N,c_N,\mu_N)$ \emph{converges in sample shape} to the algebraic measure tree $(T,c,\mu)$ if and only if for $U_1^N,U_2^N,...$ i.i.d.\ of law $\mu_N$, and $U_1,U_2,...$ i.i.d.\ of law $\mu$, for all $m\in\N$,
\begin{equation}
	\Fs_{(T_N,c_N)}(U_1^N,...,U_m^N)\underset{N\rightarrow\infty}{\Longrightarrow}\Fs_{(T,c)}(U_1,...,U_m).
\end{equation}
\label{Def:002}
\end{defi}

Since for any $m\in\N$ the space of $m$-cladograms is finite, we have the following equivalence.
\begin{proposition}
Let $(\chi_N=(T_N,c_N,\mu_N))_{N\in\N}$ and $\chi=(T,c,\mu)$ be in $\T_2$. Then $(\chi_N)_{N\in\mathbb{N}}$ converges to $\chi$ w.r.t.\ the sample shape convergence if and only if for all $m\in\N$ and $\Ft\in\overline{\FC}_m$,
\begin{equation}
	\mu_N^{\otimes m}\big(\big\{(u_1,...,u_m):\Fs_{(T_N,c_N)}(\underline{u})=\Ft\big\}\big)\underset{N\rightarrow\infty}{\longrightarrow}\mu^{\otimes m}\big(\big\{(u_1,...,u_m):\Fs_{(T,c)}(\underline{u})=\Ft\}\big).
\end{equation}
\label{p:conveq}
\end{proposition}

In what follows will consider $\alpha$-Ford trees with $N$ leaves as random algebraic measure tree which belong to the following subspace:
\begin{equation}
\label{e:002}
	\T_2^N:=\big\{(T,c,\mu)\in\T_2:\#\mathrm{lf}(T)=N\text{ and }\mu=\tfrac{1}{N}\sum_{u\in\mathrm{lf}(T)}\delta_u\big\},
\end{equation}
and let then $N$ tend to infinity.
The following proposition claims that the limit points are elements in the following closed subspace:
\begin{equation}
	\T_2^\mathrm{cont}:=\left\{(T,c,\mu)\in\T_2:\mathrm{at}(\mu)=\emptyset\right\}.
\end{equation}

\begin{proposition}[Approximations with $\T_2^\mathrm{cont}$]
Let $\chi\in\T_2$. Then $\chi\in\T_2^\mathrm{cont}$ if and only if there exists for each $N\in\N$ a $\chi_N\in\T_2^N$ such that $\chi_N\rightarrow\chi$ in one and thus all of the equivalent notions of convergence on $\T_2$ given above.
\label{p:approxT2}
\end{proposition}

\begin{proposition}[Compactness and metrizability]
$\T_2$ is a compact, metrizable space. Moreover, $\T_2^\mathrm{cont}$ is a closed subspace of $\T_2$, and thus compact as well.
\label{p:compact}
\end{proposition}

\section{Static tree models}
\label{S:Static}

Relying on the notion of sample shape convergence in $\T_2$, we can construct random algebraic measure trees by letting the number of leaves go to infinity in some finite tree models satisfying a sampling consistency property. This is what we use here to construct the family of continuum $\alpha$-Ford algebraic measure trees, together with the fact that an $m$-cladogram can be seen as an element of $\T_2^m$.

Recall the construction of the $\alpha$-Ford model on cladograms from Section \ref{S:introduction}. Recall also that we get the Comb model when $\alpha=1$, the Uniform model when $\alpha=\frac{1}{2}$, and the Yule model when $\alpha=0$, which also corresponds to the Kingman tree (see Section \ref{S:appendix} for a proof):

\begin{proposition}
For all $m\in\N$, the random cladogram obtained from the Kingman $m$-coalescent has the distribution of the $(\alpha=0)$-Ford model on $m$-cladograms.
\end{proposition}

Putting weight $1/m$ on each leaf and forgetting the labelling of the leaves, we can also define a random algebraic measure tree in $\T_2^m$. In order to distinguish $m$-cladograms and elements of $\T_2^m$, we use most of the time the letter $N$ to describe the number of leaves of algebraic measure trees.

From this construction, we thus have that, for each $\alpha\in[0,1]$ and $N\in\N$, the $\alpha$-Ford model defines a random algebraic measure tree in $\T_2^N$, that we denote by $\tau^\alpha_N$. The sequence $(\tau^\alpha_N)_N$ takes values in the compact space $\T_2$. We can therefore show that the sequence is convergent by proving that each convergent subsequence converges to the same limit. The uniqueness of limit points  results from the consistency property of the $\alpha$-Ford models:

\begin{definition}[Sampling consistency]
Consider a family $(T_N,c_N)_N$ of random algebraic trees such that $(T_N,c_N)$ has $N$ leaves. For $m\leq N$, let $U_1,...,U_m$ be a uniform random choice of $m$ distinct leaves of $\Ft_N$. We say that the family $(\Ft_N)_N$ is \emph{sampling consistent} if for all $1\leq m\leq N<\infty$, the algebraic tree associated with the shape
\begin{equation}
	\Fs_{(T_N,c_N)}(U_1,...,U_m)
\end{equation}
has the same distribution as $(T_m,c_m)$.
\end{definition}

It follows immediately from our notion of convergence that a sampling consistent family of random binary algebraic trees together with the uniform distribution $\mu_N$ on $\mathrm{lf}(T_N)$ converges weakly to a binary algebraic measure tree.

It has been shown in \cite{Ford2005} through a combinatorial argument that the $\alpha$-Ford models are \emph{deletion stable}. That is, the cladogram obtained by removing the leaf with label $m$ from the $\alpha$-Ford tree with $m$ leaves has the distribution of the $\alpha$-Ford tree with $m-1$ leaves. Furthermore, the last step of the construction of the $\alpha$-Ford cladogram assures that we have exchangeability, i.e.\ the resulting distribution on cladograms is symmetric under permutation of leaf labels. This assures that the family $(\tau_N)_N$ is sampling consistent and thus converges weakly in $\T_2$.

\begin{definition}[$\alpha$-Ford algebraic measure tree]
The \emph{$\alpha$-Ford algebraic measure tree} is the unique limit in $\T_2$ of the sequence $(\tau^\alpha_N)_N$, where $\tau^\alpha_N$ is the random algebraic measure tree in $\T_2^N$ obtained from the random $N$-cladogram distributed according to the $\alpha$-Ford model. We also give the following names for some specific values of $\alpha$:
\begin{itemize}
	\item $\alpha=0$: the Kingman algebraic measure tree,
	\item $\alpha=\frac{1}{2}$: the algebraic measure Brownian CRT,
	\item $\alpha=1$: the Comb algebraic measure tree.
\end{itemize}
\end{definition}

\begin{figure}
\setlength{\unitlength}{.28cm}
\begin{picture}(0,5)
\put(-20,0){\line(1,0){21}}
\put(3,0){\line(1,0){17}}
\put(1.1,0){$\boldsymbol{\ldots}$}
\put(-18,0){\line(0,1){2}}
\put(-16,0){\line(0,1){2}}
\put(-14,0){\line(0,1){2}}
\put(-12,0){\line(0,1){2}}
\put(-10,0){\line(0,1){2}}
\put(-8,0){\line(0,1){2}}
\put(-6,0){\line(0,1){2}}
\put(-4,0){\line(0,1){2}}
\put(-2,0){\line(0,1){2}}
\put(0,0){\line(0,1){2}}
\put(18,0){\line(0,1){2}}
\put(4,0){\line(0,1){2}}
\put(6,0){\line(0,1){2}}
\put(8,0){\line(0,1){2}}
\put(10,0){\line(0,1){2}}
\put(12,0){\line(0,1){2}}
\put(14,0){\line(0,1){2}}
\put(16,0){\line(0,1){2}}
\put(0,2){\circle*{0.3}}
\put(-18,2){\circle*{0.3}}
\put(-16,2){\circle*{0.3}}
\put(-14,2){\circle*{0.3}}
\put(-12,2){\circle*{0.3}}
\put(-10,2){\circle*{0.3}}
\put(-8,2){\circle*{0.3}}
\put(-6,2){\circle*{0.3}}
\put(-4,2){\circle*{0.3}}
\put(-2,2){\circle*{0.3}}
\put(18,2){\circle*{0.3}}
\put(4,2){\circle*{0.3}}
\put(6,2){\circle*{0.3}}
\put(8,2){\circle*{0.3}}
\put(10,2){\circle*{0.3}}
\put(12,2){\circle*{0.3}}
\put(14,2){\circle*{0.3}}
\put(16,2){\circle*{0.3}}
\put(0,0){\circle*{0.3}}
\put(-20,0){\circle*{0.3}}
\put(-18,0){\circle*{0.3}}
\put(-16,0){\circle*{0.3}}
\put(-14,0){\circle*{0.3}}
\put(-12,0){\circle*{0.3}}
\put(-10,0){\circle*{0.3}}
\put(-8,0){\circle*{0.3}}
\put(-6,0){\circle*{0.3}}
\put(-4,0){\circle*{0.3}}
\put(-2,0){\circle*{0.3}}
\put(20,0){\circle*{0.3}}
\put(4,0){\circle*{0.3}}
\put(6,0){\circle*{0.3}}
\put(8,0){\circle*{0.3}}
\put(10,0){\circle*{0.3}}
\put(12,0){\circle*{0.3}}
\put(14,0){\circle*{0.3}}
\put(16,0){\circle*{0.3}}
\put(18,0){\circle*{0.3}}
\put(0.2,2.2){$\tfrac1N$}
\put(-17.8,2.2){$\tfrac1N$}
\put(-15.8,2.2){$\tfrac1N$}
\put(-13.8,2.2){$\tfrac1N$}
\put(-11.8,2.2){$\tfrac1N$}
\put(-9.8,2.2){$\tfrac1N$}
\put(-7.8,2.2){$\tfrac1N$}
\put(-5.8,2.2){$\tfrac1N$}
\put(-3.8,2.2){$\tfrac1N$}
\put(-1.8,2.2){$\tfrac1N$}
\put(18.2,2.2){$\tfrac1N$}
\put(4.2,2.2){$\tfrac1N$}
\put(6.2,2.2){$\tfrac1N$}
\put(8.2,2.2){$\tfrac1N$}
\put(10.2,2.2){$\tfrac1N$}
\put(12.2,2.2){$\tfrac1N$}
\put(14.2,2.2){$\tfrac1N$}
\put(16.2,2.2){$\tfrac1N$}
\put(-20.5,0.6){$\tfrac1N$}
\put(20.2,0.2){$\tfrac1N$}
\end{picture}
\caption{The comb tree with $N$ leaves. }
\label{Fig:02}
\end{figure}
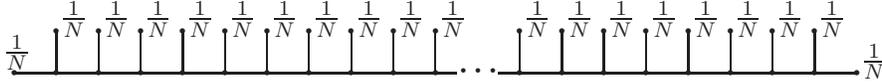

Note that as a consequence of Proposition \ref{p:approxT2}, the $\alpha$-Ford algebraic measure tree belongs to $\T_2^\mathrm{cont}$. From its definition, we know the distribution of the shape sampled by $m$ points of the $\alpha$-Ford algebraic measure tree for all $\alpha\in[0,1]$: its distribution is the $\alpha$-Ford model on $m$-cladograms.

In order to use subtree masses statistics for testing hypotheses, we need to characterize the distribution of the vector of subtree masses for the $\alpha$-Ford trees. More precisely, recall the definition of the components $\CS_v(u)$, $u,v\in T$, from \eqref{e:components} and for $\underline{u}=(u_1,u_2,u_3)\in T^3$, denote by $\underline{\eta}(\underline{u})$ the vector of the $\mu$-masses of the components of $T\setminus\{c(\underline{u})\}$, that is
\begin{equation}
	\underline{\eta}(\underline{u})=\left(\eta_i(\underline{u})\right)_{i=1,2,3}=\left(\mu(\CS_{c(\underline{u})}(u_i))\right)_{i=1,2,3}.
\end{equation}

For the case $\alpha=\frac{1}{2}$, it is known that the sample subtree masses of the algebraic measure Brownian CRT is Dirichlet distributed (see \cite[Proposition~1]{Aldous1994}, or \cite[Proposition~5.2]{LoehrMytnikWinter} for a proof in the case of algebraic measure trees using a combinatorial argument). That is, writing $\PP_\mathrm{CRT}$ for the law of the algebraic measure Brownian CRT, we have, for all $f\colon \Delta_2\rightarrow\R$ continuous bounded,
\begin{equation}
	\E_\mathrm{CRT}\left[\int_{T^3}\mu^{\otimes 3}(\mathrm{d}\underline{u}) f\left(\underline{\eta}_{(T,c,\mu)}(\underline{u})\right)\right]=\int_{\Delta_2}f(\underline{x})\mathrm{Dir}\left(\frac{1}{2},\frac{1}{2},\frac{1}{2}\right)(\mathrm{d}\underline{x}),
\end{equation}
where $\mathrm{Dir}(\frac{1}{2},\frac{1}{2},\frac{1}{2})$ is the Dirichlet distribution.

In the case $\alpha=1$, we can also write the distribution in an explicit way.

\begin{proposition}[comb tree]
Let $\mathrm{Beta}(2,2)$ be the beta distribution on $[0,1]$ and $\PP_\mathrm{Comb}$ the law of the Comb algebraic measure tree. Then for all $f\colon \Delta_2\rightarrow\R$ continuous bounded,
\begin{align*}
	\E_\mathrm{Comb}\left[\int_{T^3}\mu^{\otimes 3}(\mathrm{d}\underline{u}) f\left(\underline{\eta}_{(T,c,\mu)}(\underline{u})\right)\right]=\frac{1}{6}\sum_{\pi\in\CS_3}\int_{[0,1]}f\circ\pi^*(x,1-x,0)\mathrm{Beta}(2,2)(\mathrm{d}x),
\end{align*}
where $\CS_3$ is the set of permutations of $\{1,2,3\}$, and for $\pi\in\CS_3$, $\pi^*\colon\Delta_2\rightarrow\Delta_2$ is the induced map $\pi^*(\underline{x})=(x_{\pi(1)},x_{\pi(2)},x_{\pi(3)})$.
\end{proposition}

The arguments used for $\alpha=\frac{1}{2}$ and $\alpha=1$ do not apply to any other $\alpha$, and in particular not to the Kingman case. However, we show in Section \ref{S:applications} that for each $\alpha\in[0,1]$, the distribution of the subtree masses for the $\alpha$-Ford algebraic measure tree is invariant for a Wright-Fisher diffusion with a mutation term and catastrophies. This allows to derive an explicit representation for the Kingman algebraic measure tree.

\section{The $\alpha$-Ford chain on fixed size cladograms}
\label{S:evolving}

In this section, we introduce the $\alpha$-Ford chain as a Markov chain on the space of cladograms with a fixed number of leaves. For all $\alpha\in[0,1]$, its generator is a linear interpolation between the generator of the Aldous chain ($\alpha=\frac{1}{2}$) and the generator of what we call the Kingman chain ($\alpha=0$). The $\alpha$-Ford chain is dual to the Markov chain with reversed transition rates through a Feynman-Kac duality relation. Finally we show that the $\alpha$-Ford model is the stationary distribution of the $\alpha$-Ford chain.

For $\alpha\in[0,1]$, recall from Section \ref{S:introduction} that the $\alpha$-Ford (forward) chain is a Markov chain on the space $\FC_m$ of $m$-cladograms, defined by the following transition rate: for a pair $(k,e)$ consisting of a leaf (label) and an external (resp.\ internal) edge not adjacent to $k$ at rate $1-\alpha$ (resp. $\alpha$), the Markov chain jumps from its current state $\Ft$ to $\Ft^{(k,e)}$, which is obtained as follows (see Figures~\ref{Fig:01a} and \ref{Fig:01b}):
\begin{itemize}
	\item erase the edge (including the incident vertices) which connects $k$ to the subtree spanned by all leaves but $k$.
	\item split the remaining subtree at the edge $e$ into two pieces.
	\item reintroduce the above edge (including $k$ and the branch point) at the split point.
\end{itemize}

The transition rates of the $\alpha$-Ford (forward) chain are thus, for $\Ft,\Ft'\in\FC_m$,
\begin{equation*}
	q_\alpha^m(\Ft,\Ft')=\sum_{k\in\mathrm{lf}(\Ft)}\sum_{e\in\mathrm{edge}(\Ft_{\wedge k})}\left((1-\alpha){\bf 1}_{\mathrm{ext{\text -}edge}(\Ft_{\wedge k})}(e)+\alpha{\bf 1}_{\mathrm{int{\text -}edge}(\Ft_{\wedge k})}(e)\right){\bf 1}_{\Ft'}(\Ft^{(k,e)}),
\end{equation*}
where $\mathrm{ext{\text -}edge}(\Ft)$ (resp.\ $\mathrm{int{\text -}edge}(\Ft)$) is the set of external edges of $\Ft$ (resp.\ internal edges), and we introduced the notation
\begin{equation}
	\Ft_{\wedge k}\in\FC_{m-1}
\end{equation}
to denote the $(m-1)$-cladogram obtained from $\Ft$ by deleting the leaf with label $k$ (and relabelling the labels $j>k$ to $j-1$). The generator $\widetilde{\Omega}_\alpha^m$ of the $\alpha$-Ford chain acts on all functions $\phi:\FC_m\rightarrow\R$ as follows:
\begin{align*}
	\widetilde{\Omega}_\alpha^m\phi(\Ft)=(1-\alpha)&\sum_{k\in\mathrm{lf}(\Ft)}\sum_{e\in\mathrm{ext{\text -}edge}(\Ft_{\wedge k})}\left(\phi(\Ft^{(k,e)})-\phi(\Ft)\right)\\
	+\alpha &\sum_{k\in\mathrm{lf}(\Ft)}\sum_{e\in\mathrm{int{\text -}edge}(\Ft_{\wedge k})}\left(\phi(\Ft^{(k,e)})-\phi(\Ft)\right).
\end{align*}
We give the following names for specific values of $\alpha$:
\begin{itemize}
	\item $\alpha=0$: the Kingman (forward) chain with generator $\widetilde{\Omega}_\mathrm{Kin}^m$,
	\item $\alpha=\frac{1}{2}$: the Aldous chain $\widetilde{\Omega}_\mathrm{Ald}^m$,
	\item $\alpha=1$: the Comb (forward) chain $\widetilde{\Omega}_\mathrm{Comb}^m$.
\end{itemize}
Note that we have, for all $\alpha,\beta,\gamma\in[0,1]$ with $\beta\not =\gamma$,
\begin{equation}
\label{e:030}
   \widetilde{\Omega}_\alpha^m
 =
    \frac{\gamma-\alpha}{\gamma-\beta}\widetilde{\Omega}_{\beta}^m+\frac{\alpha-\beta}{\gamma-\beta}\widetilde{\Omega}_\gamma^m,
\end{equation}
which implies that if some results including analytical representations and duality relations hold for the $\beta$-Ford and $\gamma$-Ford chain for two particular choices of $\beta,\gamma\in[0,1]$ with $\beta\not =\gamma$ then the corresponding results can be obtained for
the $\alpha$-Ford chain for all choices of $\alpha\in[0,1]$. We therefore need to understand the model only for two different values of $\alpha\in[0,1]$. For $\alpha=\frac{1}{2}$ the model is the Aldous chain which was studied in detail in \cite{LoehrMytnikWinter}.
In this paper we choose  $\alpha=0$ as the second value and exploit
the following relation:
\begin{equation}
\label{e:interpgenm}
	\widetilde{\Omega}_\alpha^m=(1-2\alpha)\widetilde{\Omega}_\mathrm{Kin}^m+2\alpha\widetilde{\Omega}_\mathrm{Ald}^m.
\end{equation}

We are also interested in the backward Markov chain, i.e.\ the Markov chain with reversed transition rates
\begin{equation}
	q^m_{\alpha^\downarrow}(\Ft',\Ft):=q^m_\alpha(\Ft,\Ft').
\end{equation}
To describe this chain through its generator, notice that after modifying the cladogram $\Ft$ according to the forward chain, we can go back by a similar move: picking a leaf and inserting it to a given edge. But in this case, the rates for a pair $(k,e)$ will be different depending on the position of the leaf (not of the edge) in the tree. Consider for example the Kingman chain with $\alpha=0$. Since we choose any leaf and put it to an external edge in the forward chain, the edge we pick for the reverse move can be any edge (not only external edge), but the leaf we pick has to be a \emph{cherry}. We call \emph{pair of cherries} a pair of leaves which are both adjacent to the same internal branch point, and we call \emph{cherry} a leaf that belongs to a pair of cherries (Fig. \ref{f:cherries}). We write
\begin{equation}
	\mathrm{ch{\text-}lf}(\Ft)\subset\mathrm{lf}(\Ft)
\end{equation}
for the set of cherries.

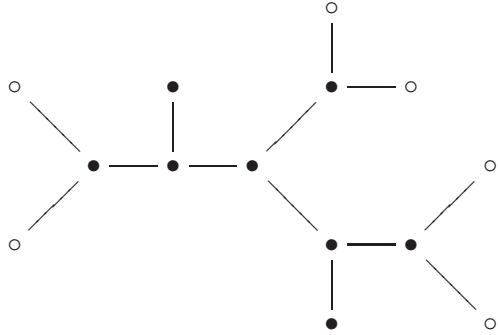
\begin{figure}[t]
\[\xymatrix@C=1.5pc@R=0.5pc{
&&&&{\circ}\ar@{-}[dd]&&\\
&&&&&&\\
{\circ}\ar@{-}[ddr]&&{\bullet}\ar@{-}[dd]&&{\bullet}\ar@{-}[ddl]&{\circ}\ar@{-}[l]&\\
&&&&&&\\
&{\bullet}\ar@{-}[r]&{\bullet}\ar@{-}[r]&{\bullet}\ar@{-}[ddr]&&&{\circ}\\
&&&&&&\\
{\circ}\ar@{-}[uur]&&&&{\bullet}\ar@{-}[r]\ar@{-}[dd]&{\bullet}\ar@{-}[uur]&\\
&&&&&&\\
&&&&{\bullet}&&{\circ}\ar@{-}[uul]}\]
\caption{An $8$-cladogram with 6 cherries.}\label{f:cherries}
\end{figure}

Thus, the generator $\widetilde{\Omega}_{\alpha^\downarrow}^m$ of the \emph{$\alpha$-Ford backward chain} acts on functions $\phi:\FC_m\rightarrow\R$ as follows:
\begin{equation}
\label{e:026}
\begin{aligned}
	\widetilde{\Omega}_{\alpha^\downarrow}^m\phi(\Ft):=(1-\alpha)&\sum_{k\in\mathrm{ch{\text -}lf}(\Ft)}\sum_{e\in\mathrm{edge}(\Ft_{\wedge k})}\left(\phi(\Ft^{(k,e)})-\phi(\Ft)\right)\\
	+\alpha&\sum_{k\notin\mathrm{ch{\text -}lf}(\Ft)}\sum_{e\in\mathrm{edge}(\Ft_{\wedge k})}\left(\phi(\Ft^{(k,e)})-\phi(\Ft)\right).
\end{aligned}
\end{equation}
Note that the $\alpha$-Ford chain is symmetric if and only if $\alpha=\frac{1}{2}$ (Aldous chain).

We have the following relation between $\widetilde{\Omega}_\alpha^m$ and $\widetilde{\Omega}_{\alpha^\downarrow}^m$:

\begin{proposition}[Feynman-Kac duality]
Let $X^{m,\alpha}=(X^{m,\alpha}_t)_{t\geq 0}$ be the $\alpha$-Ford forward chain and $Y^{m,\alpha}=(Y^{m,\alpha}_t)_{t\geq 0}$ the backward chain. Then for all $\Ft,\Fs\in\FC_m$,
\begin{equation}
	\E_\Fs\left[{\bf 1}_\Ft(X^{m,\alpha}_t)\right]=\E_\Ft\left[{\bf 1}_\Fs(Y^{m,\alpha}_t)\exp\left(\int_0^t\beta^m_\alpha(Y^{m,\alpha}_s)\mathrm{d}s\right)\right],
\end{equation}
where $\beta^m_\alpha(\Ft):=(1-2\alpha)\big(\#(\mathrm{ch{\text -}lf}(\Ft))(2m-5)-m(m-1)\big)$. The second factor in $\beta^m_\alpha(\Ft)$ is the difference between the number of possible moves for the Kingman backward and forward chains.
\label{P:Feynman}
\end{proposition}

\begin{proof}
For $m\in\N$ and $\Ft,\Ft'\in\FC_m$, we write $H(\Ft',\Ft):={\bf 1}\{\Ft'=\Ft\}$. We claim that
\begin{equation}
\label{e:backfor}
	\widetilde{\Omega}_\alpha^mH(\cdot,\Ft)(\Ft')=\widetilde{\Omega}_{\alpha^\downarrow}^mH(\Ft',\cdot)(\Ft)+\beta^m_\alpha(\Ft)H(\Ft',\cdot)(\Ft),
\end{equation}
which is true for $\alpha=\frac{1}{2}$, since the Aldous chain is symmetric and $\beta^m_\frac{1}{2}(\Ft)=0$ for all $\Ft$. In the case $\alpha=0$ we have,
\begin{align*}
	\widetilde{\Omega}_\mathrm{Kin}^m &H(\cdot,\Ft)(\Ft')-\widetilde{\Omega}_{\mathrm{Kin}^\downarrow}^mH(\Ft',\cdot)(\Ft)\\
	&=\sum_{k\in\mathrm{lf}(\Ft')}\sum_{e\in\mathrm{ext{\text -}edge}(\Ft'_{\wedge k})}\left({\bf 1}_\Ft(\Ft'^{(k,e)})-{\bf 1}_\Ft(\Ft')\right)-\sum_{l\in\mathrm{ch{\text -}lf}(\Ft)}\sum_{f\in\mathrm{edge}(\Ft_{\wedge l})}\left({\bf 1}_{\Ft'}(\Ft^{(l,f)})-{\bf 1}_{\Ft'}(\Ft)\right)\\
	&=-\sum_{k\in\mathrm{lf}(\Ft')}\sum_{e\in\mathrm{ext{\text -}edge}(\Ft'_{\wedge k})}{\bf 1}_\Ft(\Ft')+\sum_{l\in\mathrm{ch{\text -}lf}(\Ft)}\sum_{f\in\mathrm{edge}(\Ft_{\wedge l})}{\bf 1}_{\Ft'}(\Ft)\\
	&=\beta^m_0(\Ft){\bf 1}_{\Ft'}(\Ft),
\end{align*}
where for the second equality we used that if there exists one forward move to go from $\Ft'$ to $\Ft$, then there exists one backward move to go from $\Ft$ to $\Ft'$, and reciprocally. And if this is the case, then both moves are unique.

Using \eqref{e:interpgenm} and an analogous relation for backward chains, we have \eqref{e:backfor}. The result then follows by \cite[Lemma 4.4.11, Corollary 4.4.13]{EthierKurtz86}
\end{proof}

We have defined a family of Markov chains on a finite state space. For all $\alpha\in[0,1)$, the chain is irreducible recurrent and thus has a unique invariant distribution. But the following result stills holds for $\alpha=1$.

\begin{proposition}
For all $\alpha\in[0,1]$ and $m\in\N$, the $\alpha$-Ford model on $m$-cladograms is the unique invariant distribution of the $\alpha$-Ford (forward) chain. In particular, for all $\phi\colon\FC_m\rightarrow\R$,
\begin{equation}\label{e:invdistm}
	\sum_{\Ft\in\FC_m}\widetilde{\PP}_\mathrm{Ford}^{\alpha,m}(\Ft)\widetilde{\Omega}_\alpha^m\phi(\Ft)=0,
\end{equation}
where $\widetilde{\PP}_\mathrm{Ford}^{\alpha,m}$ denotes the law of the $\alpha$-Ford model on $m$-cladograms.
\end{proposition}

\begin{proof}
Let $\alpha\in[0,1]$. We need only to prove the result for $\phi$ of the form ${\bf 1}_{\Ft'}$ with $\Ft'\in\FC_m$. Thus, we can rewrite \eqref{e:invdistm} as follows
\begin{equation}
	m(m-1-3\alpha)\widetilde{\PP}_\mathrm{Ford}^{\alpha,m}(\Ft')=\sum_{\Ft\in\FC_m}\widetilde{\PP}_\mathrm{Ford}^{\alpha,m}(\Ft)q^m_\alpha(\Ft,\Ft').
\end{equation}
Writing $X^m$ for the random $m$-cladogram distributed according to the $\alpha$-Ford model, we want to prove the equality:
\begin{equation}\label{eq:invdist}
	m(m-1-3\alpha)\PP(X^m=\Ft')=\sum_{\Ft\in\FC_m}\PP(X^m=\Ft)q^m_\alpha(\Ft,\Ft').
\end{equation}

For the right-hand side of the equation, we use the consistency property of the $\alpha$-Ford model. We have
\begin{equation}
\begin{aligned}
	\sum_{\Ft\in\FC_m}&\PP(X^m=\Ft)q^m_\alpha(\Ft,\Ft')\\
	&= \sum_{\Ft\in\FC_m}\PP(X^m=\Ft)\sum_{k\in\mathrm{lf}(\Ft)}\left((1-\alpha)\sum_{e\in\mathrm{ext{\text -}edge}(\Ft_{\wedge k})}{\bf 1}_{\Ft'}(\Ft^{(k,e)})+\alpha\sum_{e\in\mathrm{int{\text -}edge}(\Ft_{\wedge k})}{\bf 1}_{\Ft'}(\Ft^{(k,e)})\right)\\
	&= \sum_{k=1}^m\sum_{\substack{\Ft\in\FC_m\\ \Ft_{\wedge k}={\Ft'}_{\wedge k}}}\PP(X^m=\Ft)\left((1-\alpha)\sum_{e\in\mathrm{ext{\text -}edge}(\Ft_{\wedge k})}{\bf 1}_{\Ft'}(\Ft^{(k,e)})+\alpha\sum_{e\in\mathrm{int{\text -}edge}(\Ft_{\wedge k})}{\bf 1}_{\Ft'}(\Ft^{(k,e)})\right),
\end{aligned}
\end{equation}
because if ${\bf 1}_{\Ft'}(\Ft^{(k,e)})=1$, then $\Ft_{\wedge k}={\Ft'}_{\wedge k}$. Now, if $\Ft_{\wedge k}={\Ft'}_{\wedge k}$, then there exists $e\in\mathrm{ext{\text -}edge}(\Ft_{\wedge k})$ such that ${\bf 1}_{\Ft'}(\Ft^{(k,e)})=1$ only if $k\in\mathrm{ch{\text -}lf}(\Ft)$. In this case, the edge $e$ is unique. The same still holds for internal edges and non-cherry leaves. Furthermore, due to the consistency property of the $\alpha$-Ford model, we can write $\PP((X^m)_{\wedge k}={\Ft'}_{\wedge k})=\PP(X^{m-1}={\Ft'}_{\wedge k})$. Therefore,
\begin{equation}
	\sum_{\Ft\in\FC_m}\PP(X^m=\Ft)q^m_\alpha(\Ft,\Ft') = \sum_{k=1}^m\PP(X^{m-1}={\Ft'}_{\wedge k})\left((1-\alpha){\bf 1}_{\{k\in\mathrm{ch{\text -}lf}(\Ft')\}}+\alpha{\bf 1}_{\{k\notin\mathrm{ch{\text -}lf}(\Ft')\}}\right).
\end{equation}

For the left-hand side of equation \eqref{eq:invdist}, we use that we can obtain an $m$-cladogram with distribution the $\alpha$-Ford model as follows:
\begin{itemize}
	\item take an $(m-1)$-cladogram $z$ with distribution the $\alpha$-Ford model,
	\item pick an edge $e$ of $z$ randomly according to the weights of the $\alpha$-Ford model,
	\item insert a leaf labelled $k$ together with an edge at $e$, and denote this new $m$-cladogram by $z^e$,
	\item apply a uniform permutation $\sigma$ to the leaf labels of $z^e$. We write $\sigma(z^e)$ for the new $m$-cladogram.
\end{itemize}
Therefore, writing $\CS_m$ for the set of permutations of $\{1,...,m\}$,
\begin{equation}
\begin{aligned}
	\PP(X^m=\Ft') = & \sum_{z\in\FC_{m-1}}\PP(X^{m-1}=z)\frac{1}{m!}\sum_{\sigma\in\CS_m}\frac{1}{m-1-3\alpha}\\
	&\times\left((1-\alpha)\sum_{e\in\mathrm{ext{\text -}edge}(z)}{\bf 1}_{\Ft'}(\sigma(z^e))+\alpha\sum_{e\in\mathrm{int{\text -}edge}(z)}{\bf 1}_{\Ft'}(\sigma(z^e))\right)\\
	= & \sum_{z\in\FC_{m-1}}\PP(X^{m-1}=z)\frac{1}{m}\sum_{k=1}^m\frac{1}{(m-1)!}\sum_{\substack{\sigma\in\CS_m\\ \sigma(m)=k}}\frac{{\bf 1}_{{\Ft'}_{\wedge k}}(\sigma_{|\{1,...,m-1\}}(z))}{m-1-3\alpha}\\
	&\times\left((1-\alpha)\sum_{e\in\mathrm{ext{\text -}edge}(z)}{\bf 1}_{\Ft'}(\sigma(z^e))+\alpha\sum_{e\in\mathrm{int{\text -}edge}(z)}{\bf 1}_{\Ft'}(\sigma(z^e))\right).
\end{aligned}
\end{equation}
We used that if $\sigma\in\CS_m$ is such that $\sigma(m)=k$ and ${\bf 1}_\Ft(\sigma(z^e))=1$, then $(\sigma(z^e))_{\wedge_k}=\sigma_{|\{1,...,m-1\}}(z)$. Now, as for the right-hand side, if $\sigma(m)=k$ and $\sigma_{|\{1,...,m-1\}}(z)=\Ft_{\wedge k}$, then there exists $e\in\mathrm{ext{\text -}edge}(z)$ such that ${\bf 1}_\Ft(\sigma(z^{m,e}))=1$ only if $k\in\mathrm{ch{\text -}lf}(\Ft)$. In this case, the edge $e$ is unique, and this also holds for internal edges and non-cherry leaves. Thus we have
\begin{align*}
	\PP(X^m=\Ft') = & \frac{1}{m}\sum_{k=1}^m\frac{1}{(m-1)!}\sum_{\substack{\sigma\in\CS_m\\ \sigma(m)=k}}\sum_{z\in\FC_{m-1}}\PP\left(\sigma_{|\{1,...,m-1\}}(X^{m-1})=\sigma_{|\{1,...,m-1\}}(z)\right)\\
	& \times{\bf 1}_{{\Ft'}_{\wedge k}}(\sigma_{|\{1,...,m-1\}}(z))\frac{(1-\alpha){\bf 1}_{\{k\in\mathrm{ch{\text -}lf}(\Ft')\}}+\alpha{\bf 1}_{\{k\notin\mathrm{ch{\text -}lf}(\Ft')\}}}{m-1-3\alpha}\\
	= & \frac{1}{m}\sum_{k=1}^m\frac{1}{(m-1)!}\sum_{\substack{\sigma\in\CS_m\\ \sigma(m)=k}}\PP(X^{m-1}={\Ft'}_{\wedge k})\frac{(1-\alpha){\bf 1}_{\{k\in\mathrm{ch{\text -}lf}(\Ft')\}}+\alpha{\bf 1}_{\{k\notin\mathrm{ch{\text -}lf}(\Ft')\}}}{m-1-3\alpha}\\
	= & \frac{1}{m}\sum_{k=1}^m\PP(X^{m-1}={\Ft'}_{\wedge k})\frac{(1-\alpha){\bf 1}_{\{k\in\mathrm{ch{\text -}lf}(\Ft')\}}+\alpha{\bf 1}_{\{k\notin\mathrm{ch{\text -}lf}(\Ft')\}}}{m-1-3\alpha},
\end{align*}
where the second equality results again from the consistency property of the $\alpha$-Ford model. This gives \eqref{eq:invdist}, and thus the result.
\end{proof}

A binary tree in $\T_2^N$ can be seen as an $N$-cladogram with the uniform measure on leaves, and without leaf labels. Therefore, we can consider all the above Ford chains on $\T_2^N$ and we denote in this case the generators by $\Omega$ instead of $\widetilde{\Omega}$. For example, $\Omega_\alpha^N$ is the generator of the $\alpha$-Ford forward chain on $\T_2^N$. Using this idea, we let the number of leaves go to infinity and study the diffusion limit.

\section{The $\alpha$-Ford chain in the diffusion limit}
\label{S:diffusion limit}

When the state space is finite, a Markov chain is well-defined by its generator. But we now want to consider the diffusion approximation of the Ford chains on $\T_2^N$ as the number of leaves $N$ goes to infinity. As it is, we show here that the generators $(\Omega_\alpha^N)_N$ converge in a uniform way and that the martingale problem associated to the limit generator is well posed. To prove the uniqueness of the solution of the martingale problem, we use a Feynman-Kac duality result. Furthermore, the unique solution has continuous paths in $\T_2^\mathrm{cont}$, which we call the $\alpha$-Ford diffusion. Finally we show that the $\alpha$-Ford algebraic measure tree is an invariant distribution.

We introduce here the operator which is the limit of the generators $(\Omega_\alpha^N)_N$ (see Proposition \ref{p:convgen}). To define it, we use shape polynomials as test functions. Recall that $\Fs_{(T,c)}(\underline{u})$ denotes the shape spanned by $m$ points $\underline{u}=(u_1,...,u_m)$ sampled from the tree $(T,c)$.

\begin{defi}[Shape polynomials]
A \emph{shape polynomial} is a linear combination of functions $\Phi^{m,\Ft}:\T_2\rightarrow\R$ of the form
\begin{equation}
	\Phi^{m,\Ft}(\chi):=\mu^{\otimes m}(\Fs_{(T,c)}^{-1}(\Ft))=\int_{T^m}\mu^{\otimes m}(\mathrm{d}\underline{u}){\bf 1}_\Ft(\Fs_{(T,c)}(\underline{u})),
\end{equation}
where $\chi=(T,c,\mu)$, $m\in\N$ and $\Ft\in\overline{\FC}_m$. We write $\Pi_\Fs$ for the set of all shape polynomials.
\end{defi}

In other words, for $\Ft\in\FC_m$ and $\chi=(T,c,\mu)$, $\Phi^{m,\Ft}(\chi)$ describes the probability that the cladogram obtained by $m$ points sampled from $T$ w.r.t.\ $\mu$ is the cladogram $\Ft$.

The set $\Pi_\Fs$ is an algebra. With Proposition \ref{p:conveq}, it is contained in the space $\CC(\T_2)$ of continuous functions and it separates the points of $\T_2$. Therefore, since $\T_2$ is compact by Proposition \ref{p:compact}, it is dense in $\CC(\T_2)$ by the theorem of Stone-Weierstrass.

Since the infinite trees we consider are limits as $N\rightarrow\infty$ of trees in $\T_2^N$, Proposition \ref{p:approxT2} provides that we can limit our work to binary algebraic measure trees in $\T_2^\mathrm{cont}$. Now consider $m\in\N$ and $\Ft\in\overline{\FC}_m\setminus\FC_m$. Then if $\Fs_{(T,c)}(u_1,...,u_m)=\Ft$ for some $\chi=(T,c,\mu)\in\T_2^\mathrm{cont}$, we have that $u_1,...,u_m$ are not distinct, so that $\Phi^{m,\Ft}(\chi)=0$ because $\mathrm{at}(\mu)=\emptyset$. For this reason, we will limit the domain of the operator of the $\alpha$-Ford forward chain $\CD(\Omega_\alpha)$ to shape polynomials using $m$-cladograms instead of $m$-labelled cladograms:
\begin{equation}
	\CD(\Omega_\alpha):=\mathrm{span}\{\Phi^{m,\Ft}:m\in\N,\Ft\in\FC_m\},
\end{equation}
which is also dense in $\CC(\T_2)$.

We can now define the operator $\Omega_\alpha$ which acts on shape polynomials as follows:
\begin{equation}
	\Omega_\alpha\Phi^{m,\Ft}(\chi):=\int_T\mu^{\otimes m}(\mathrm{d}\underline{u})\widetilde{\Omega}_\alpha^m{\bf 1}_\Ft(\Fs_{(T,c)}(\underline{u})),
\end{equation}
where ${\bf 1}_\Ft$ plays the role of the test function for $\widetilde{\Omega}_\alpha^m$. We can then extend this definition by linearity to $\CD(\Omega_\alpha)$. Here again, we denote $\Omega_\alpha$ by $\Omega_\mathrm{Kin}$, $\Omega_\mathrm{Ald}$, $\Omega_\mathrm{Comb}$ when $\alpha=0, \frac{1}{2}, 1$ respectively, and we note that we have, for all $\alpha\in[0,1]$,
\begin{equation}\label{e:interpgen}
	\Omega_\alpha=(1-2\alpha)\Omega_\mathrm{Kin}+2\alpha\Omega_\mathrm{Ald},
\end{equation}
which we will use recurrently for the proofs.

\begin{proposition}
For all $\Phi\in\CD(\Omega_\alpha)$, we have $\Omega_\alpha\Phi\in\CD(\Omega_\alpha)$. In particular,
\begin{equation}
	(\Phi,\Omega_\alpha\Phi)\in\CC(\T_2)\times\CC(\T_2).
\end{equation}
\end{proposition}
\begin{proof}
For $\Phi\in\CD(\Omega_\alpha)$, $\Phi$ and $\Omega_\alpha\Phi$ are shape polynomials, hence continuous by definition of sample shape convergence.
\end{proof}

To prove the existence, we will want to use, as in \cite{LoehrMytnikWinter}, an argument similar to \cite[Lemma 4.5.1]{EthierKurtz86}. For this, we use that $\Omega_\alpha$ is the limit of generators which each defines a martingale problem with solutions. This limit is uniform in the following sense:

\begin{proposition}[Convergence of generators]\label{p:convgen}
Let $\alpha\in[0,1]$. For all $\Phi\in\CD(\Omega_\alpha)$, we have
\begin{equation}
	\lim_{N\rightarrow\infty}\sup_{\chi\in\T_2^N}\left|\Omega_\alpha^N\Phi(\chi)-\Omega_\alpha\Phi(\chi)\right|=0.
\end{equation}
\end{proposition}

\begin{proof}
Since the result was shown by \cite{LoehrMytnikWinter} for the Aldous case $\alpha=\frac{1}{2}$, and using \eqref{e:interpgen}, we need only to show it for the Kingman case $\alpha=0$.

Consider $\Phi\in\CD(\Omega_\alpha)$. By linearity, we can assume w.l.o.g.\ that $\Phi=\Phi^{m,\Ft}$ for some $m\in\N$ and $\Ft\in\FC_m$. If $m=1,2,3$, there is only one cladogram in $\FC_m$, so that $\Phi^{m,\Ft}$ is constant on $\T_2^N$, for each $N\in\N$. Therefore, $\Omega^N_\mathrm{Kin}\Phi^{m,\Ft}(\chi)=0$ for all $\chi\in\T_2^N$, and the convergence holds since $\Omega_\mathrm{Kin}\Phi^{m,\Ft}$ also equals to zero for $m=1,2,3$. Thus, we suppose $m\geq 4$. Fix $N\in\N$ and $\chi=(T,c,\mu)\in\T_2^N$. We write
\begin{equation}
	\epsilon:=\frac{1}{N}.
\end{equation}
We extend the algebraic tree to allow for potential new branch points and new leaves due to the chain moves on binary trees. To this end, for each edge $e\in\mathrm{edge}(T,c)$, we introduce two additional points $x_e$, $y_e$, i.e., we consider
\begin{equation}
	\overline{T}=T\cup\bigcup_{e\in\mathrm{edge}(T,c)}\{x_e,y_e\},
\end{equation}
and extend $c$ to $\overline{c}:\overline{T}^3\rightarrow\overline{T}$ which is uniquely defined as follows (Fig. \ref{f:extendedtree}). $(\overline{T},\overline{c})$ is an algebraic tree such that for $e=\{a,b\}\in\mathrm{edge}(T,c)$, we have $x_e\in[a,b]$ in $(\overline{T},\overline{c})$, and
\begin{equation}
	\overline{c}(y_e,x_e,z)=x_e,~~\forall z\in\overline{T}\setminus\{y_e\}.
\end{equation}

\begin{figure}[t]
\[
\xymatrix@C=0.8pc@R=0.8pc{
	{\bullet}\ar@{-}[ddrr] &&                  && {\bullet}\ar@{-}[dd] &&       && {\bullet} 		&&&&	{\bullet}\ar@{-}[dr] && {\circ}\ar@{-}[dl]  && {\bullet}\ar@{-}[d] &&       && {\bullet} \\
                           &&&&&&&&                                                             &&&&        & {\circ}\ar@{-}[dr] && &{\circ}\ar@{-}[d]&{\circ}\ar@{-}[l] &&{\circ}\ar@{-}[ur]\ar@{-}[dr] & \\
	                   && {\bullet}\ar@{-}[rr] && {\bullet}\ar@{-}[rr] && {\bullet}\ar@{-}[uurr] &&     &&&&          {\circ} && {\bullet}\ar@{-}[r] &{\circ}\ar@{-}[r]& {\bullet}\ar@{-}[r] &{\circ}\ar@{-}[r]& {\bullet}\ar@{-}[ur] &&{\circ} \\
                           &&&&&&&&                                                             &&&&        & {\circ}\ar@{-}[ur]\ar@{-}[ul]^(1){y_e} &&{\circ}\ar@{-}[u] && {\circ}\ar@{-}[u]&& {\circ}\ar@{-}[ul] & \\
	{\bullet}\ar@{-}[uurr]^{e} &&              &&                  &&      && {\bullet}\ar@{-}[uull]&&&&        {\bullet}\ar@{-}[ur]_(1){x_e} &&                  &&                  && {\circ}\ar@{-}[ur] && {\bullet}\ar@{-}[ul]
}\]
\caption{A finite algebraic tree $(T,c)$ and the extended tree $(\overline{T},\overline{c})$.}\label{f:extendedtree}
\end{figure}
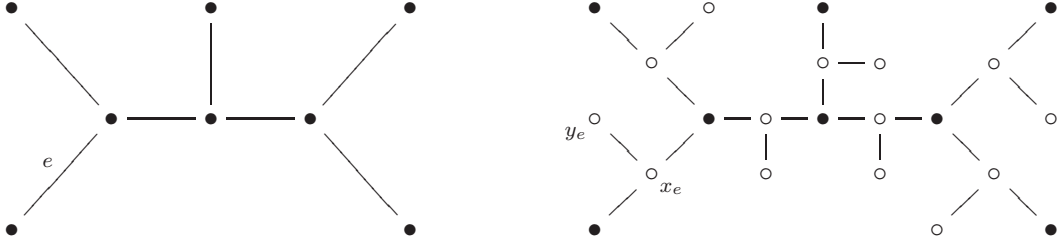

For $k\in\{1,...,m\}$ and $x\in\overline{T}$, let $\theta_{k,x}:T^m\rightarrow\overline{T}^m$ be the replacement operator which replaces the $k^{\mathrm{th}}$-coordinate by $x$. For $\chi=(T,\mu)=(\overline{T},\mu)$ and $(x,e)\in\mathrm{lf}(T,c)\times\mathrm{ext{\text -}edge}(T,c)$, we write $\chi^{(x,e)}$ the binary algebraic measure tree obtained by the chain move with $z$, i.e.,
\begin{equation}
	\chi^{(x,e)}:=(\overline{T},\overline{c},\mu+\epsilon\delta_{y_e}-\epsilon\delta_x).
\end{equation}
The difference between sampling with the new and old measure is given by
\begin{equation}
\begin{aligned}
	(\mu+\epsilon\delta_{y_e}&-\epsilon\delta_x)^{\otimes m}-\mu^{\otimes m}\\
	= &~ \epsilon\sum_{k=1}^m\mu^{\otimes (k-1)}\otimes(\delta_{y_e}-\delta_x)\otimes\mu^{\otimes (m-k)}\\
	& +\epsilon^2\sum_{1\leq k<j\leq m}\mu^{\otimes (k-1)}\otimes(\delta_{y_e}-\delta_x)\otimes\mu^{\otimes (j-1)}\otimes(\delta_{y_e}-\delta_x)\otimes\mu^{\otimes (m-j-k)}+\widetilde{\mu}\\
	= &~ \epsilon\sum_{k=1}^m(\mu^{\otimes m}\circ\theta^{-1}_{k,y_e}-\mu^{\otimes m}\circ\theta^{-1}_{k,x})-\epsilon^2\sum_{j\neq k=1}^m\mu^{\otimes m}\circ\theta^{-1}_{k,y_e}\circ\theta^{-1}_{j,x}+\widetilde{\mu},
\end{aligned}
\end{equation}
where $\widetilde{\mu}$ is a signed measure on $\overline{T}^m$ with $\widetilde{\mu}\{(u_1,...,u_m):u_1,...,u_m\text{ distinct}\}=0$. But since $\Ft\in\FC_m$, the leaf labels are distinct. Thus,
\begin{equation}
	\Omega_\mathrm{Kin}^N\Phi^{m,\Ft}(\chi)=\sum_{x\in\mathrm{lf}(T,c)}\sum_{e\in\mathrm{ext{\text -}edge}(T,c)}(\Phi^{m,\Ft}(\chi^{(x,e)})-\Phi^{m,\Ft}(\chi))=:\sum_{k=1}^m A_k-\sum_{k\neq j=1}^m B_{k,j},
\end{equation}
with
\begin{equation}
	A_k=\epsilon\sum_{x\in\mathrm{lf}(T,c)}\sum_{e\in\mathrm{ext{\text -}edge}(T,c)}\int_{T^m}\mu^{\otimes m}(\mathrm{d}\underline{u})\left({\bf 1}_\Ft(\Fs_{(\overline{T},\overline{c})}(\theta_{k,y_e}\underline{u}))-{\bf 1}_\Ft(\Fs_{(\overline{T},\overline{c})}(\theta_{k,x}\underline{u}))\right),
\end{equation}
and
\begin{equation}
	B_{k,j}=\epsilon^2\sum_{x\in\mathrm{lf}(T,c)}\sum_{e\in\mathrm{ext{\text -}edge}(T,c)}\int_{T^m}\mu^{\otimes m}(\mathrm{d}\underline{u}){\bf 1}_\Ft(\Fs_{(\overline{T},\overline{c})}(\theta_{k,y_e}\circ\theta_{j,x}\underline{u})).
\end{equation}

Recall the notation $\Ft_{\wedge k}\in\FC_{m-1}$ for the $(m-1)$-cladogram obtained from $\Ft$ by deleting the leaf with label $k$ (and relabelling the labels $j>k$ to $j-1$), i.e., if $\Ft=\Fs_{(\overline{T},\overline{c})}(\underline{u})$, then $\Ft_{\wedge k}=\Fs_{(\overline{T},\overline{c})}(\underline{u}_{\wedge k})$ with $\underline{u}_{\wedge k}=(u_1,...,u_{k-1},u_{k+1},...,u_m)$.

Furthermore, for $\underline{u}\in T^m$, we define
\begin{equation}
	E_{\Ft,k}(\underline{u})=\{v\in\overline{T}:\Fs_{(\overline{T},\overline{c})}(\theta_{k,v}\underline{u})=\Ft\}.
\end{equation}
Note that $E_{\Ft,k}(\underline{u})$ does not depend on $u_k$ and that $E_{\Ft,k}(\underline{u})\neq\emptyset$ only if $\Ft_{\wedge k}=\Fs_{(\overline{T},\overline{c})}(\underline{u}_{\wedge k})$. In this case, $\Ft_{\wedge k}=\Fs_{(\overline{T},\overline{c})}(\underline{u}_{\wedge k})$ "corresponds" to an edge of $\Ft_{\wedge k}$.
Let
\begin{equation}
	\nu:=\epsilon\sum_{e\in\mathrm{ext{\text -}edge}(T,c)}\delta_{y_e}
\end{equation}
be the uniform distribution on $\{y_e:e\in\mathrm{ext{\text -}edge}(T,c)\}$. By Fubini's theorem, we have that
\begin{equation}
\begin{aligned}
	A_k & = \int_{T^m}\mu^{\otimes m}(\mathrm{d}\underline{u})\left(\sum_{e\in\mathrm{ext{\text -}edge}(T,c)}{\bf 1}_\Ft(\Fs_{(\overline{T},\overline{c})}(\theta_{k,y_e}\underline{u}))-\sum_{x\in\mathrm{lf}(T,c)}{\bf 1}_\Ft(\Fs_{(\overline{T},\overline{c})}(\theta_{k,x}\underline{u}))\right)\\
	& = \epsilon^{-1}\int_{T^m}\mu^{\otimes m}(\mathrm{d}\underline{u})(\nu(E_{\Ft,k}(\underline{u}))-\mu(E_{\Ft,k}(\underline{u})))
\end{aligned}
\end{equation}
and
\begin{equation}
\begin{aligned}
	B_{k,j} & = \epsilon\sum_{e\in\mathrm{ext{\text -}edge}(T,c)}\int_T\mu(dx)\int_{T^m}\mu^{\otimes m}(\mathrm{d}\underline{u}){\bf 1}_\Ft(\Fs_{(\overline{T},\overline{c})}(\theta_{k,y_e}\circ\theta_{j,x}\underline{u}))\\
	& = \epsilon\int_{T^m}\mu^{\otimes m}(\mathrm{d}\underline{u})\sum_{e\in\mathrm{ext{\text -}edge}(T,c)}{\bf 1}_\Ft(\Fs_{(\overline{T},\overline{c})}(\theta_{k,y_e}\underline{u}))\\
	& = \int_{T^m}\mu^{\otimes m}(\mathrm{d}\underline{u})\nu(E_{\Ft,k}(\underline{u}))
\end{aligned}
\end{equation}

To go further in the calculation of $A_k$ and $B_{k,j}$, we want to find a relation between $\nu(E_{\Ft,k}(\underline{u}))$ and $\mu(E_{\Ft,k}(\underline{u}))$. To have $E_{\Ft,k}(\underline{u})\neq\emptyset$ we need $\Fs_{(\overline{T},\overline{c})}(\underline{u}_{\wedge k})=\Ft_{\wedge k}$. If it is the case, $N\nu(E_{\Ft,k}(\underline{u}))$ and $N\mu(E_{\Ft,k}(\underline{u}))$ are respectively given by the number of external edges and the number of leaves "between" two specific points $z_{k,\Ft,T},z'_{k,\Ft,T}\in T$ which correspond through $\Fs_{(\overline{T},\overline{c})}$ to the vertices in $\Ft$ that are neighbours of $j$ where $j$ is such that $\{j,k\}$ is an edge of $\Ft$ (see Fig. \ref{f:zkzkp}). For $z,z'\in T$,
\begin{align*}
	\#\{e\in\mathrm{ext{\text -}edge}(T,c)&:\overline{c}(y_e,z,z')\in(z,z')\}\\
	&=\#\{x\in\mathrm{lf}(T,c):\overline{c}(x,z,z')\in(z,z')\}+{\bf 1}_{\mathrm{lf}(T)}(z)+{\bf 1}_{\mathrm{lf}(T)}(z'),
\end{align*}
since if $z\in\mathrm{lf}(T)$, the edge adjacent to $z$ is included in the left-hand side of the inequality but $z$ is not included in the right-hand side and the same holds for $z'$.

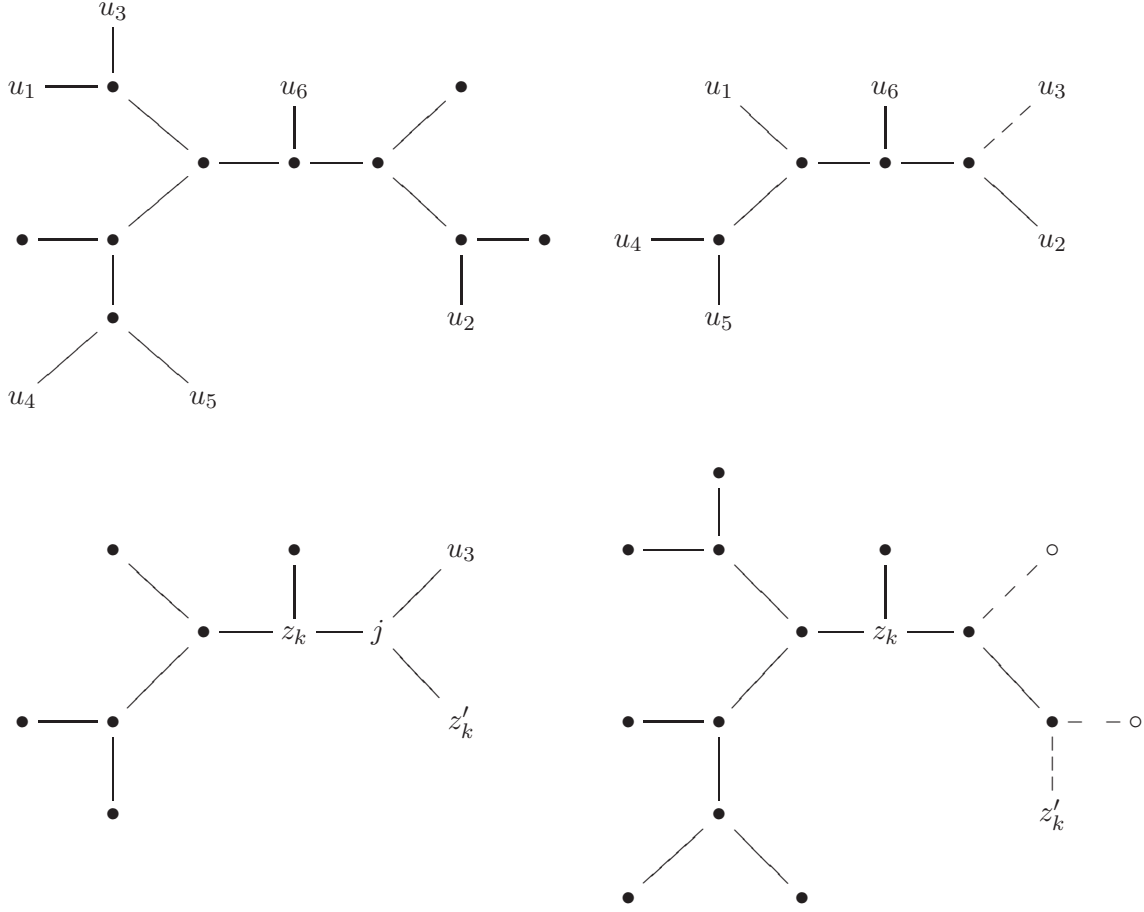
\begin{figure}[t]
\[
\xymatrix@=1.4pc{
&{u_3}\ar@{-}[d]&&&&
&
&&&&&\\
u_1\ar@{-}[r]&{\bullet}\ar@{-}[dr]&&u_6\ar@{-}[d]&&{\bullet}\ar@{-}[dl]&
&
&u_1\ar@{-}[dr]&&u_6\ar@{-}[d]&&u_3\ar@{--}[dl]\\
&&{\bullet}\ar@{-}[r]&{\bullet}\ar@{-}[r]&{\bullet}&&
&
&&{\bullet}\ar@{-}[r]&{\bullet}\ar@{-}[r]&{\bullet}&\\
{\bullet}\ar@{-}[r]&{\bullet}\ar@{-}[ur]&&&&{\bullet}\ar@{-}[ul]&{\bullet}\ar@{-}[l]
&
u_4\ar@{-}[r]&{\bullet}\ar@{-}[ur]&&&&u_2\ar@{-}[ul]\\
&{\bullet}\ar@{-}[u]&&&&u_2\ar@{-}[u]&
&
&u_5\ar@{-}[u]&&&&\\
u_4\ar@{-}[ur]&&u_5\ar@{-}[ul]&&&&
&
&&&&&\\
&&&&&&&
&{\bullet}\ar@{-}[d]&&&&&\\
&{\bullet}\ar@{-}[dr]&&{\bullet}\ar@{-}[d]&&u_3\ar@{-}[dl]&&
{\bullet}\ar@{-}[r]&{\bullet}\ar@{-}[dr]&&{\bullet}\ar@{-}[d]&&{\circ}\ar@{--}[dl]&\\
&&{\bullet}\ar@{-}[r]&z_k\ar@{-}[r]&j&&&
&&{\bullet}\ar@{-}[r]&z_k\ar@{-}[r]&{\bullet}&&\\
{\bullet}\ar@{-}[r]&{\bullet}\ar@{-}[ur]&&&&z'_k\ar@{-}[ul]&&
{\bullet}\ar@{-}[r]&{\bullet}\ar@{-}[ur]&&&&{\bullet}\ar@{-}[ul]&{\circ}\ar@{--}[l]\\
&{\bullet}\ar@{-}[u]&&&&&&
&{\bullet}\ar@{-}[u]&&&&z'_k\ar@{--}[u]&\\
&&&&&&&
{\bullet}\ar@{-}[ur]&&{\bullet}\ar@{-}[ul]&&&&}\]
\caption{Here, $m=6$, $k=3$, with $T$ on the top left side and $\Ft$ on the top right side. We have that $\Ft_{\wedge3}=\Fs_{(\overline{T},\overline{c})}(\underline{u}_{\wedge3})$, and $u_2\in\{z_{k,\Ft,T},z'_{k,\Ft,T}\}$, so that ${\bf 1}_{\mathrm{lf}(T)}(z_{k,\Ft,T})+{\bf 1}_{\mathrm{lf}(T)}(z'_{k,\Ft,T})=1$.}\label{f:zkzkp}
\end{figure}

Therefore, we need to distinguish cases depending on the position of $k$ in $\Ft$, i.e.\ whether it is a cherry or not. Recall that $m\geq 4$. If $k\notin \mathrm{ch{\text -}lf}(\Ft)$, then $z_{k,\Ft,T},z'_{k,\Ft,T}\notin\mathrm{lf}(T)$, so that
\begin{equation}
	\nu(E_{\Ft,k}(\underline{u}))=\mu(E_{\Ft,k}(\underline{u}))
\end{equation}
and
\begin{equation}
\begin{aligned}
	A_k & = 0,\\
	B_{k,j} & = \Phi^{m,\Ft}(\chi).
\end{aligned}
\end{equation}
If $k\in \mathrm{ch{\text -}lf}(\Ft)$, then exactly one of $z_{k,\Ft,T}$ and $z'_{k,\Ft,T}$ is in $\mathrm{lf}(T)$, such that
\begin{equation}
	N\nu(E_{\Ft,k}(\underline{u}))-N\mu(E_{\Ft,k}(\underline{u}))={\bf 1}_{\Ft_{\wedge k}}(\Fs_{(\overline{T},\overline{c})}(\underline{u}_{\wedge k}))
\end{equation}
and
\begin{equation}
\begin{aligned}
	A_k & = \Phi^{m-1,\Ft_{\wedge k}}(\chi),\\
	B_{k,j} & = \Phi^{m,\Ft}(\chi)+\epsilon\Phi^{m-1,\Ft_{\wedge k}}(\chi).
\end{aligned}
\end{equation}
Therefore, for $m\geq4$, we have
\begin{equation}
	\Omega_\mathrm{Kin}^N\Phi^{m,\Ft}(\chi)=\sum_{k\in \mathrm{ch{\text -}lf}(\Ft)}\left(\Phi^{m-1,\Ft_{\wedge k}}(\chi)-\epsilon(m-1)\Phi^{m-1,\Ft_{\wedge k}}(\chi)\right)-m(m-1)\Phi^{m,\Ft}(\chi).
\end{equation}
Since
\begin{equation}
	\left|\sum_{k\in \mathrm{ch{\text -}lf}(\Ft)}\epsilon(m-1)\Phi^{m-1,\Ft_{\wedge k}}(\chi)\right|\leq\epsilon m(m-1),
\end{equation}
goes to $0$ as $N\rightarrow\infty$, we focus on the two other terms. We have
\begin{equation*}
	\sum_{k\in \mathrm{ch{\text -}lf}(\Ft)}\Phi^{m-1,\Ft_{\wedge k}}(\chi) = \int_{T^m}\mu^{\otimes m}(\mathrm{d}\underline{u})\sum_{k\in \mathrm{ch{\text -}lf}(\Ft)}{\bf 1}_{\Ft_{\wedge k}}(\Fs_{(T,c)}(\underline{u}_{\wedge k})).
\end{equation*}
For $\underline{u}\in T^m$, we have $\Fs_{(T,c)}(\underline{u}_{\wedge k})=\Ft_{\wedge k}$ if and only if there is an edge $e$ of $\Ft_{\wedge k}$ such that $\Fs_{(T,c)}(\underline{u})=\Ft^{(k,e)}$, where $\Ft^{(k,e)}$ is the $m$-cladogram obtained by inserting a leaf with label $k$ at the edge $e$ in $\Ft_{\wedge k}$ (and relabelling the labels $j\geq k$ to $j+1$). If such an edge $e$ exists, it is unique, and we have
\begin{equation}
	{\bf 1}_{\Ft_{\wedge k}}(\Fs_{(T,c)}(\underline{u}_{\wedge k}))=\sum_{e\in\mathrm{edge}(\Ft_{\wedge k})}{\bf 1}_{\Ft^{(k,e)}}(\Fs_{(T,c)}(\underline{u})).
\end{equation}
Recall (\ref{e:026}) the generator $\widetilde{\Omega}_{\mathrm{Kin}^\downarrow}^m$
of the backward Kingman chain. By linearity, we have
\begin{equation}
\begin{aligned}
\label{e:027}
	\sum_{k\in \mathrm{ch{\text -}lf}(\Ft)}\Phi^{m-1,\Ft_{\wedge k}}(\chi)
	&= \int_{T^m}\mu^{\otimes m}(\mathrm{d}\underline{u})\sum_{k\in \mathrm{ch{\text -}lf}(\Ft)}\sum_{e\in\mathrm{edge}(\Ft_{\wedge k})}{\bf 1}_{\Fs_{(T,c)}(\underline{u})}(\Ft^{(k,e)})\\
	&= \int_{T^m}\mu^{\otimes m}(\mathrm{d}\underline{u})\left(\widetilde{\Omega}_{\mathrm{Kin}^\downarrow}^m{\bf 1}_{\Fs_{(T,c)}(\underline{u})}(\Ft)+\#(\mathrm{ch{\text -}lf}(\Ft))\#\mathrm{edge}(\Ft_{\wedge k}){\bf 1}_{\Fs_{(T,c)}(\underline{u})}(\Ft)\right).
\end{aligned}
\end{equation}
With the notation $\beta_0^m(\Ft):=\#(\mathrm{ch{\text -}lf}(\Ft))(2m-5)-m(m-1)$, we have
\begin{equation}
\begin{aligned}
	\Omega_\mathrm{Kin}^N\Phi^{m,\Ft}(\chi)=&\int_{T^m}\mu^{\otimes m}(\mathrm{d}\underline{u})\left(\widetilde{\Omega}_{\mathrm{Kin}^\downarrow}^m{\bf 1}_{\Fs_{(T,c)}(\underline{u})}(\Ft)+\beta^m_0(\Ft){\bf 1}_{\Fs_{(T,c)}(\underline{u})}(\Ft)\right)\\
	&+\sum_{k\in \mathrm{ch{\text -}lf}(\Ft)}\epsilon(m-1)\Phi^{m-1,\Ft_{\wedge k}}(\chi).
\end{aligned}
\end{equation}
Finally, using the relation between the Kingman forward and backward chains stated in \eqref{e:backfor} with $\alpha=0$,
\begin{align*}
	\Omega_\mathrm{Kin}^N\Phi^{m,\Ft}(\chi) = & \int_{T^m}\mu^{\otimes m}(\mathrm{d}\underline{u})\widetilde{\Omega}_\mathrm{Kin}^m{\bf 1}_\Ft(\Fs_{(T,c)}(\underline{u}))+\sum_{k\in \mathrm{ch{\text -}lf}(\Ft)}\epsilon(m-1)\Phi^{m-1,\Ft_{\wedge k}}(\chi)\\
	= & \Omega_\mathrm{Kin}\Phi^{m,\Ft}(\chi)+\sum_{k\in \mathrm{ch{\text -}lf}(\Ft)}\epsilon(m-1)\Phi^{m-1,\Ft_{\wedge k}}(\chi),
\end{align*}
where the last term goes to $0$ as $N\rightarrow\infty$ uniformly over all $\chi\in\T_2^N$. We thus have the result for all $m\geq4$.
\end{proof}

We can now deduce from this convergence the two following results.

\begin{corollary}[The limiting martingale problem]\label{limmartprob}
Let $\alpha\in[0,1]$. Let $(\chi_N)_{N\in\N}$ be a sequence of random binary algebraic measure trees with $\chi_N\in\T_2^N$, such that
\begin{equation}
	\chi_N\Rightarrow\chi,~\text{ as }N\rightarrow\infty,
\end{equation}
where $\chi$ is a random tree in $\T_2^\mathrm{cont}$ with distribution $P_0$. Let $X^n:=(X_t^N)_{t\geq0}$ be the $\alpha$-Ford forward chain started in $\chi_N$. Then the sequence $(X^N)_{N\in\N}$ is tight in $\CD_{\T_2}$, and any limit point $(X_t)_{t\geq0}$ has continuous paths in $\T_2^\mathrm{cont}$ and satisfies the $(\Omega_\alpha,\CD(\Omega_\alpha),P_0)$-martingale problem.
\end{corollary}

\begin{corollary}[Existence of a solution]\label{exismartprob}
Let $\alpha\in[0,1]$. For any probability measure $P_0$ on $\T_2^\mathrm{cont}$ there exists a solution in $\CC_{\T_2^\mathrm{cont}}(\R_+)$ to the $(\Omega_\alpha,\CD(\Omega_\alpha),P_0)$-martingale problem.
\end{corollary}

To prove the uniqueness of the solution for the martingale problem, we will use a result of duality which appeared in the proof of Proposition \ref{p:convgen}. We claim that we have a duality between a diffusion and a Markov chain on a finite state space: the dual of the $\alpha$-Ford diffusion is the dual $\alpha$-Ford Markov chain, that is, the backward chain.

\begin{proposition}[Feynman-Kac duality]
Let $\alpha\in[0,1]$. Let $P_0$ be a probability measure on $\T_2^{\mathrm{cont}}$, let $X:=((T_t,c_t,\mu_t))_{t\geq0}$ be a solution to the $(\Omega_\alpha,\CD(\Omega_\alpha),P_0)$-martingale problem in $\CD_{\T_2^{\mathrm{cont}}}(\R_+)$. For $m\in\N$ and $\Ft\in\FC_m$, we denote by $Y^m:=(Y_t^m)_{t\geq0}$ the $\alpha$-Ford backward chain on $\FC_m$-cladograms started in $Y_0^m=\Ft$. Then if $Y^m$ is independent of $X$,
\begin{equation}\label{e:duality}
	\E_{P_0}^X\left[\Phi^{m,\Ft}(X_t)\right]=\int_{\T_2^\mathrm{cont}}\E_\Ft^{Y^m}\left[\Phi^{m,Y_t^m}(\chi)\exp\left(\int_0^t\beta^m_\alpha(Y_s^m)\mathrm{d}s\right)\right]P_0(\mathrm{d}\chi),
\end{equation}
where $\beta^m_\alpha(\Ft)=(1-2\alpha)\left(\#(\mathrm{ch{\text -}lf}(\Ft))(2m-5)-m(m-1)\right)$.
\end{proposition}

\begin{proof}
Consider $m\in\N$. For $\chi=(T,c,\mu)\in\T_2^\mathrm{cont}$ and $\Ft\in\FC_m$, we define $f(\chi,\Ft):=\Phi^{m,\Ft}(\chi)$. With this notation and using \eqref{e:backfor}, we can write
\begin{align*}
	\Omega_\alpha f(\cdot,\Ft)(\chi) := & \int_{T^m}\mu^{\otimes m}(\mathrm{d}\underline{u})\widetilde{\Omega}_\alpha^m{\bf 1}_\Ft(\Fs_{(T,c)}(\underline{u}))\\
	= & \int_{T^m}\mu^{\otimes m}(\mathrm{d}\underline{u})\left(\widetilde{\Omega}_{\alpha^\downarrow}^m{\bf 1}_{\Fs_{(T,c)}(\underline{u})}(\Ft)+\beta^m_\alpha(\Ft){\bf 1}_{\Fs_{(T,c)}(\underline{u}}(\Ft)\right)\\
	= & \widetilde{\Omega}_{\alpha^\downarrow}^m\left(\int_{T^m}\mu^{\otimes m}(\mathrm{d}\underline{u}){\bf 1}_{\Fs_{(T,c)}(\underline{u})}\right)(\Ft)+\beta^m_\alpha(\Ft)\Phi^{m,\Ft}(\chi)\\
	= & \widetilde{\Omega}_{\alpha^\downarrow}^mf(\chi,\cdot)(\Ft)+\beta^m_\alpha(\Ft)f(\chi,\Ft).
\end{align*}
The result then follows by \cite[Lemma 4.4.11]{EthierKurtz86} and \cite[Corollary 4.4.13]{EthierKurtz86}.
\end{proof}

\begin{proposition}[Uniqueness of the martingale problem]\label{uniqmartprob}
Let $\alpha\in[0,1]$. For all probability measures $P_0$ on $\T_2^\mathrm{cont}$, uniqueness holds for the $(\Omega_\alpha,\CD(\Omega_\alpha),P_0)$-martingale problem in $\CD_{\T_2^\mathrm{cont}}(\R_+)$.
\end{proposition}

Summing up the results of Corollaries \ref{limmartprob} and \ref{exismartprob}, and Proposition \ref{uniqmartprob}, we have shown Theorem \ref{t:martprob}. The following result gives the existence of an invariant distribution, namely the $\alpha$-Ford algebraic measure tree.

\begin{proposition}
Let $\alpha\in[0,1]$. The $\alpha$-Ford algebraic measure tree is an invariant distribution of the $\alpha$-Ford diffusion, that is, for all $m\in\N$ and $\Ft\in\FC_m$,
\begin{equation}
	\int_{\T_2^\mathrm{cont}}\PP_\mathrm{Ford}^\alpha(\mathrm{d}\chi)\Omega_\alpha\Phi^{m,\Ft}(\chi)=0,
\end{equation}
where $\PP_\mathrm{Ford}^\alpha$ is the law of the $\alpha$-Ford algebraic measure tree on $\T_2^\mathrm{cont}$.
\end{proposition}

\begin{proof}
We fix $m\in\N$ and $\Ft\in\FC_m$. We have that
\begin{equation}
	\E_\mathrm{Ford}^\alpha\left[\Phi^{m,\Ft}(\chi)\right]=\E_\mathrm{Ford}^\alpha\left[\mu^{\otimes m}\{\underline{u}:\Fs_{(T,c)}(\underline{u})=\Ft\}\right]=\widetilde{\PP}_\mathrm{Ford}^{\alpha,m}(\Ft),
\end{equation}
where $\widetilde{\PP}_\mathrm{Ford}^{\alpha,m}$ denotes the law of the $\alpha$-Ford model on $m$-cladograms. Therefore

\begin{align*}
	\int_{\T_2^\mathrm{cont}}\PP_\mathrm{Ford}^\alpha(\mathrm{d}\chi)\Omega_\alpha\Phi^{m,\Ft}(\chi) = & \int_{\T_2^\mathrm{cont}}\PP_\mathrm{Ford}^\alpha(\mathrm{d}\chi)\int_{T^m}\mu^{\otimes m}(\mathrm{d}\underline{u})\widetilde{\Omega}_\alpha^m{\bf 1}_\Ft(\Fs_{(T,c)}(\underline{u}))\\
	= & \int_{\FC_m}\widetilde{\PP}_\mathrm{Ford}^{\alpha,m}(\mathrm{d}y)\widetilde{\Omega}_\alpha^m{\bf 1}_\Ft(y)=0,
\end{align*}
since the $\alpha$-Ford model on $m$-cladograms is the invariant distribution of the $\alpha$-Ford chain.
\end{proof}

A natural question to ask is then whether the $\alpha$-Ford diffusion converges toward the $\alpha$-Ford invariant distribution, for any initial distribution. This result would give in the same time the uniqueness of the invariant distribution. However, the difficulty to show this result come from the exponential term in the duality equation \eqref{e:duality} and therefore it remains an open question.

\section{Application on sample subtree masses}
\label{S:applications}

In this section we are interested in the infinitesimal evolution of the law of the vector of subtree masses under the $\alpha$-Ford diffusion. Recall the definition of the subtree masses $\underline{\eta}(\underline{u})$ for $\underline{u}=(u_1,u_2,u_3)\in T^3$, that is,
\begin{equation}
	\underline{\eta}(\underline{u})=\left(\eta_i(\underline{u})\right)_{i=1,2,3}=\left(\mu(\CS_{c(\underline{u})}(u_i))\right)_{i=1,2,3}.
\end{equation}
Recall also from \eqref{e:Phi} that for mass polynomials of degree 3
\begin{equation}
\label{e:Phi2}
	\Phi^f(\chi)=\int_{T^3}\mu^{\otimes 3}(\mathrm{d}\underline{u})f\left(\underline{\eta}(\underline{u})\right)
\end{equation}
where $f\in\CC^2([0,1]^3)$ and $\chi = (T,c,\mu)\in\T_2$, we extend the generator of the $\alpha$-Ford diffusion by defining
\begin{align*}
	\Omega_\alpha\Phi^f(\chi)=& \int\mu^{\otimes 3}(\mathrm{d}\underline{u})\left(\sum_{i,j=1}^{3}\eta_i(\delta_{ij}-\eta_j)\partial^2_{ij}f(\underline{\eta}(\underline{u}))+(2-\alpha)\sum_{i=1}^3(1-3\eta_i)\partial_i f(\underline{\eta}(\underline{u}))\right.\\
	& +\frac{\alpha}{2}\sum_{i\neq j=1}^3 \frac{{\bf 1}_{\eta_i\neq 0}}{\eta_i}\left(f\circ\theta_{i,j}(\underline{\eta}(\underline{u}))-f(\underline{\eta}(\underline{u}))\right)+\frac{\alpha}{2}\sum_{i\neq j=1}^3\left({\bf 1}_{\eta_j=0}-{\bf 1}_{\eta_i=0}\right)\partial_if(\underline{\eta}(\underline{u}))\\
	& \left.+(2-3\alpha)\sum_{i=1}^3\left(f(e_i)-f(\underline{\eta}(\underline{u}))\right)\right).
\end{align*}

\begin{proposition}
Let $\alpha\in[0,1]$. For all test functions $\Phi^f$ of the form \eqref{e:Phi2} with $f\colon[0,1]^3\rightarrow\R$ twice continuously differentiable,
\begin{equation}
	\lim_{N\rightarrow\infty}\sup_{\chi_N\in\T_2^N}\left|\Omega_\alpha^N\Phi^f(\chi)-\Omega_\alpha\Phi^f(\chi)\right|=0.
\end{equation}
\end{proposition}

\begin{proof}
The result has already been proved by \cite{LoehrMytnikWinter} for the Aldous case $\alpha=\frac{1}{2}$ so it is enough to show it for the Kingman case. Consider $f\in\CC^2([0,1]^3)$, and
recall the notations introduced in the proof of Proposition~\ref{p:convgen}. For each permutation $\pi$ of $\{1,2,3\}$, define $\pi_*\colon\Delta_3\rightarrow\Delta^3$ by $\pi_*(\underline{\eta})=(\eta_{\pi(1)},\eta_{\pi(2)},\eta_{\pi(3)})$. Since $\Phi^f=\Phi^{f\circ\pi_*}$ and $\Omega_\mathrm{Kin}\Phi^f=\Omega_\mathrm{Kin}\Phi^{f\circ\pi_*}$ for every permutation of $\{1,2,3\}$, we may and do assume w.l.o.g.\ that $f$ is symmetric.

Let $\Ft$ be the only $3$-cladogram. Since $\mathrm{at}(\mu)=\emptyset$, we can introduce a term ${\bf 1}_{\Ft}\left(\Fs_{(T,c)}(\underline{u})\right)$, which we do for later purpose. We thus have
\begin{align*}
	\Omega_\mathrm{Kin}^N\Phi^f(\chi)= & \Omega_\mathrm{Kin}^N\int_{T^3}\mu^{\otimes 3}(\mathrm{d}\underline{u}){\bf 1}_{\Ft}\left(\Fs_{(T,c)}(\underline{u})\right)f\left(\underline{\eta}_\chi(\underline{u})\right)\\
	:=& \sum_{x\in\mathrm{lf}(T,c)}\sum_{e\in\mathrm{ext{\text -}edge}(T,c)}\left(\int_{\overline{T}^3}(\mu^z)^{\otimes 3}(\mathrm{d}\underline{u}){\bf 1}_{\Ft}\left(\Fs_{(\overline{T},\overline{c})}(\underline{u})\right)f\left(\underline{\eta}_{\chi^{(x,e)}}(\underline{u})\right)\right.\\
	&-\left.\int_{T^3}\mu^{\otimes 3}(\mathrm{d}\underline{u}){\bf 1}_{\Ft}\left(\Fs_{(T,c)}(\underline{u})\right)f\left(\underline{\eta}_\chi(\underline{u})\right)\right)\\
	=& \sum_{x\in\mathrm{lf}(T,c)}\sum_{e\in\mathrm{ext{\text -}edge}(T,c)}\left(\int_{\overline{T}^3}((\mu^z)^{\otimes 3}-\mu^{\otimes 3})(\mathrm{d}\underline{u}){\bf 1}_{\Ft}\left(\Fs_{(\overline{T},\overline{c})}(\underline{u})\right)f\left(\underline{\eta}_{\chi^{(x,e)}}(\underline{u})\right)\right)\\
	&+\sum_{x\in\mathrm{lf}(T,c)}\sum_{e\in\mathrm{ext{\text -}edge}(T,c)}\left(\int_{T^3}\mu^{\otimes 3}(\mathrm{d}\underline{u})\left(f(\underline{\eta}_{\chi^{(x,e)}}(\underline{u}))-f(\underline{\eta}_\chi(\underline{u}))\right)\right)\\
=:&\Delta_\mu+\Delta_f.
\end{align*}

The term $\Delta_f$ appears when the measure $\mu$ is left unchanged, but there is change in the three masses. It gives the Wright-Fisher term. To see this, fix $\underline{u}\in T^3$. We abbreviate $\eta_i=\eta_i(\underline{u})$ as long as $\underline{u}$ is fixed. We denote the components of $T\setminus\{c(\underline{u})\}$ by $S_i$, $i=1,2,3$, ordered such that $\eta_i=\mu(S_i)$. For all $z=(x,e)\in\mathrm{Kin}^\uparrow(T,c)$ with $x\in S_i$ and $e\in S_j$, we have by a Taylor expansion that
\begin{equation}
	f(\underline{\eta}_{\chi^{(x,e)}}(\underline{u}))-f(\underline{\eta}_\chi(\underline{u}))=\left(\epsilon(\partial_j-\partial_i)+\frac{\epsilon^2}{2}(\partial^2_{ii}+\partial^2_{jj}-2\partial^2_{ij}))\right)f(\underline{\eta})+o(\epsilon^2),
\end{equation}
and the $o(\epsilon^2)$-term is uniform in the binary trees with $N$ leaves as $N\rightarrow\infty$. Now summing over all $z\in\mathrm{Kin}^\uparrow(T,c)$, we have
\begin{align*}
	\sum_{x\in\mathrm{lf}(T,c)}\sum_{e\in\mathrm{ext{\text -}edge}(T,c)}&\left(f(\underline{\eta}_{\chi^{(x,e)}}(\underline{u}))-f(\underline{\eta}_\chi(\underline{u}))\right)\\
	&=\sum_{i\neq j=1}^{2m-3}\frac{\eta_i}{\epsilon}\frac{\eta_j}{\epsilon}\epsilon\left((\partial_j-\partial_i+\frac{\epsilon}{2}(\partial^2_{ii}+\partial^2_{jj}-2\partial^2_{ij}))f(\underline{\eta})+o(\epsilon)\right)\\
	&=\sum_{i\neq j=1}^{2m-3}\eta_i\eta_j(\partial^2_{ii}-\partial^2_{ij})f(\underline{\eta})+o(1)\\
	&=\sum_{i, j=1}^{2m-3}\eta_i(\delta_{ij}-\eta_j)\partial^2_{ij}f(\underline{\eta})+o(1).
\end{align*}
where we used for the second equality that the highest order term is anti-symmetric in $i\neq j$.

Finally, Fubini's Theorem gives
\begin{equation*}\label{e:deltaf}
	\Delta_f=\int_{T^m}\mu^{\otimes m}(\mathrm{d}\underline{u}){\bf 1}_{\Ft}(\Fs_{(T,c)}(\underline{u}))\sum_{i, j=1}^{2m-3}\eta_i(\delta_{ij}-\eta_j)\partial^2_{ij}f(\underline{\eta})+o(1).
\end{equation*}

The term $\Delta_\mu$ gives the effect of the change in $\mu$, when the subtree masses are considered after the chain move. We use for this the same decomposition as in the proof of Proposition \ref{p:convgen}. For $k=1,2,3$ and $x\in\overline{T}$, recall that $\theta_{k,x}:T^3\rightarrow\overline{T}^3$ is the replacement operator which replaces the $k^{\mathrm{th}}$-coordinate by $x$. The difference between sampling with the new and old measure is given by
\begin{equation*}
	(\mu^z)^{\otimes 3}-\mu^{\otimes 3}	= \epsilon\sum_{k=1}^3(\mu^{\otimes 3}\circ\theta^{-1}_{k,y_e}-\mu^{\otimes 3}\circ\theta^{-1}_{k,x})-\epsilon^2\sum_{j\neq k=1}^3\mu^{\otimes 3}\circ\theta^{-1}_{k,y_e}\circ\theta^{-1}_{j,x}+\widetilde{\mu},
\end{equation*}
where $\widetilde{\mu}$ is a signed measure on $\overline{T}^3$ with $\widetilde{\mu}\{(u_1,u_2,u_3):u_1,u_2,u_3\text{ distinct}\}=0$. But since $\Ft\in\FC_3$, the leaf labels are distinct. The purpose of introducing ${\bf 1}_{\Ft}\left(\Fs_{(T,c)}(\underline{u})\right)$ is to be able to ignore the term due to $\widetilde{\mu}$. Thus,
\begin{equation}
	\Delta_\mu=:\sum_{k=1}^3 A_k-\sum_{k\neq j=1}^3 B_{k,j},
\end{equation}
with
\begin{align*}
	A_k= & \epsilon\sum_{x\in\mathrm{lf}(T,c)}\sum_{e\in\mathrm{ext{\text -}edge}(T,c)}\left(\int_{T^3}\mu^{\otimes 3}(\mathrm{d}\underline{u}){\bf 1}_\Ft(\Fs_{(\overline{T},\overline{c})}(\theta_{k,y_e}\underline{u}))f(\underline{\eta}_{\chi^{(x,e)}}(\theta_{k,y_e}\underline{u}))\right.\\
	& -\left.\int_{T^3}\mu^{\otimes 3}(\mathrm{d}\underline{u}){\bf 1}_\Ft(\Fs_{(\overline{T},\overline{c})}(\theta_{k,x}\underline{u}))f(\underline{\eta}_{\chi^{(x,e)}}(\theta_{k,x}\underline{u}))\right),
\end{align*}
and
\begin{equation*}
	B_{k,j}=\epsilon^2\sum_{x\in\mathrm{lf}(T,c)}\sum_{e\in\mathrm{ext{\text -}edge}(T,c)}\int_{T^3}\mu^{\otimes 3}(\mathrm{d}\underline{u}){\bf 1}_\Ft(\Fs_{(\overline{T},\overline{c})}(\theta_{k,y_e}\circ\theta_{j,x}\underline{u}))f(\underline{\eta}_{\chi^{(x,e)}}(\theta_{k,y_e}\circ\theta_{j,x}\underline{u})).
\end{equation*}

We calculate first $B_{k,j}$. Let $k\neq j$. Note that the element $u_j$ sampled according to $\mu$ does not appear in the integrand, but is replaced by the leaf $x$ which is sampled according to
\begin{equation*}
	\epsilon\sum_{x\in\mathrm{lf}(T)}\delta_x = \mu.
\end{equation*}
Thus, using first Fubini's theorem, we can write
\begin{equation*}
	B_{k,j}=\epsilon\int_{T^3}\mu^{\otimes 3}(\mathrm{d}\underline{u})\sum_{e\in\mathrm{ex{\text -}edge}(T,c)}{\bf 1}_\Ft(\Fs_{(\overline{T},\overline{c})}(\theta_{k,y_e}\underline{u}))f(\underline{\eta}_{\chi^(u_j,e)}(\theta_{k,y_e}\underline{u})).
\end{equation*}
In the same way, the element $u_k$ does not appear in the integrand. However, it is this time replaced by the leaf $y_e$ which is sampled according to
\begin{equation*}
	\nu:=\epsilon\sum_{e\in\mathrm{ex-edge}(T)}\delta_{y_e} \neq \mu.
\end{equation*}
But we can still find a similar relation. Indeed, consider $(u_1,u_2,u_3)$ sampled from $(T,c)$ according to $\mu$. One can notice that if $e\notin\{e_{u_i},i=1,2,3\}$, where $e_u$ denotes the external edge connected to the leaf $u$, then
\begin{equation*}
	{\bf 1}_\Ft(\Fs_{(\overline{T},\overline{c})}(\theta_{k,y_e}\underline{u}))f(\underline{\eta}_{\chi^(u_j,e)}(\theta_{k,y_e}\underline{u}))={\bf 1}_\Ft(\Fs_{(\overline{T},\overline{c})}(\theta_{k,u}\underline{u}))f(\underline{\eta}_{\chi^(u_j,e_u)}(\theta_{k,u}\underline{u})),
\end{equation*}
where $u\in\mathrm{lf}(T)$ is such that $e_u=e$. In other words, sampling according to $\nu$ instead of $\mu$ leaves the above quantity unchanged when $e\notin\{e_{u_i},i=1,2,3\}$. This is also true when $e=e_{u_k}$. But if $e=e_{u_i}$ for some $i\in\{1,2,3\}\setminus\{k\}$, then
\begin{equation*}
	{\bf 1}_\Ft(\Fs_{(\overline{T},\overline{c})}(\theta_{k,u_i}\underline{u}))=0,
\end{equation*}
because $\Ft\in\FC_3$ has distinct leaf labels, but not $\Fs_{(\overline{T},\overline{c})}(\theta_{k,u_i}\underline{u})$. Therefore using that $\mathrm{at}(\mu)=\emptyset$, we can write
\begin{align*}
	B_{k,j}=&\int_{T^3}\mu^{\otimes 3}(\mathrm{d}\underline{u})f(\underline{\eta}_{\chi^(u_j,e_{u_k})}(\underline{u}))+\epsilon\int_{T^3}\mu^{\otimes 3}(\mathrm{d}\underline{u})\sum_{\substack{i=1\\ i\neq k}}^3f(\underline{\eta}_{\chi^(u_j,e_{u_i})}(\theta_{k,y_{e_{u_i}}}\underline{u}))\\
	=&\int_{T^3}\mu^{\otimes 3}(\mathrm{d}\underline{u})f(\eta_j-\epsilon,\eta_k+\epsilon,\eta_i))+O(\epsilon)\\
	=&\int_{T^3}\mu^{\otimes 3}(\mathrm{d}\underline{u})f(\underline{\eta}_\chi(\underline{u}))+O(\epsilon),
\end{align*}
where we used for the second equality that $f$ is bounded on $[0,1]^3$, and for the last equality a Taylor expansion and that $f'$ is also bounded on $[0,1]^3$. We proceed in the same way to calculate $A_k$: since $\mathrm{at}(\mu)=\emptyset$,
\begin{align*}
	A_k=& \int_{T^3}\mu^{\otimes 3}(\mathrm{d}\underline{u})\sum_{x\in\mathrm{lf}(T,c)}f(\underline{\eta}_{\chi^{(x,e_{u_k})}}(\underline{u})) +\epsilon\int_{T^3}\mu^{\otimes 3}(\mathrm{d}\underline{u})\sum_{x\in\mathrm{lf}(T,c)}\sum_{\substack{i=1\\ i\neq k}}^3f(\underline{\eta}_{\chi^{(x,e_{u_i})}}(\theta_{k,y_{e_{u_i}}}\underline{u}))\\
	& -\int_{T^3}\mu^{\otimes 3}(\mathrm{d}\underline{u})\sum_{e\in\mathrm{ex-edge}(T,c)}f(\underline{\eta}_{\chi^{(u_k,e)}}(\underline{u}))\\
	=& \int_{T^3}\mu^{\otimes 3}(\mathrm{d}\underline{u})\sum_{i=1}^3\frac{\eta_i}{\epsilon}(f(\eta_i-\epsilon,\eta_k+\epsilon,\eta_j)-f(\eta_i+\epsilon,\eta_k-\epsilon,\eta_j))+a_k(\epsilon)\\
	=& \int_{T^3}\mu^{\otimes 3}(\mathrm{d}\underline{u})2\sum_{i=1}^3\eta_i(\partial_k-\partial_i) f(\underline{\eta}_\chi(\underline{u}))+a_k(\epsilon)+O(\epsilon),
\end{align*}
where $a_k(\epsilon)$ denotes the correction term due to a difference between the sampling of leaves and the sampling of external edges:
\begin{align*}
	a_k(\epsilon):=&\epsilon\int_{T^3}\mu^{\otimes 3}(\mathrm{d}\underline{u})\sum_{x\in\mathrm{lf}(T,c)}\sum_{\substack{i=1\\ i\neq k}}^3f(\underline{\eta}_{\chi^{(x,e_{u_i})}}(\theta_{k,y_{e_{u_i}}}\underline{u}))\\
	=&\int_{T^3}\mu^{\otimes 3}(\mathrm{d}\underline{u})\sum_{\substack{i=1\\ i\neq k}}^3\epsilon\left(\sum_{x\in\mathrm{lf}(T,c)\setminus\{u_i\}}f(\epsilon,\epsilon,1-2\epsilon)+f(0,\epsilon,1-\epsilon)\right)\\
	=&\int_{T^3}\mu^{\otimes 3}(\mathrm{d}\underline{u})2f(0,0,1)+O(\epsilon).
\end{align*}

In order to calculate $\sum_{k=1}^3 A_k$, we use that
\begin{align*}
	\sum_{k=1}^3\sum_{i=1}^3\eta_i(\partial_k-\partial_i) f(\underline{\eta}_\chi(\underline{u})) =& \sum_{k=1}^3\partial_k f(\underline{\eta}_\chi(\underline{u}))-3 \sum_{i=1}^3\eta_i\partial_i f(\underline{\eta}_\chi(\underline{u}))\\
	=& \sum_{i=1}^3(1-3\eta_i)\partial_k f(\underline{\eta}_\chi(\underline{u})).
\end{align*}
Adding up the results for $\Delta_f$ \eqref{e:deltaf} and $\Delta_\mu$, the proposition is proved for $\alpha=0$, and so for all $\alpha\in[0,1]$.
\end{proof}

We are now interested in the subtree masses distribution of the limit process, which the $\alpha$-Ford algebraic measure tree. Thus we are looking at the moments of this distribution. For $\underline{k}=(k_1,k_2,k_3)\in\N^3$, we define $f^{(\underline{k})}\colon\Delta_2\rightarrow \R$ by
\begin{equation}
	f^{(\underline{k})}(\underline{\eta})=\eta_1^{k_1}\eta_2^{k_2}\eta_3^{k_3}.
\end{equation}

Writing $\underline{X}^\alpha_\infty$ for the subtree masses of the $\alpha$-Ford algebraic measure tree, we have for all $\alpha\in[0,1]$ that $\E[f^{(0,0,0)}(\underline{X}^\alpha_\infty)]=1$ and
\begin{equation}
	\E[f^{(1,0,0)}(\underline{X}^\alpha_\infty)]=\E[f^{(0,1,0)}(\underline{X}^\alpha_\infty)]=\E[f^{(0,0,1)}(\underline{X}^\alpha_\infty)]=\frac{1}{3}.
\end{equation}
Moreover, the following recursive relations hold:
\begin{lemma}[Moments of subtree mass distribution of $\alpha$-Ford] For all $\alpha\in[0,1]$ and $\underline{k}\in\mathbb{N}_0^3$,
\begin{equation}
\begin{aligned}
	\E\big[f^{(\underline{k})}(\underline{X}^\alpha_\infty)\big]
 &=
   \frac{1}{(S+3)(S+2-3\alpha)}\Big(\sum_{i=1}^3\mathbf{1}_{\{k_i\not =0\}}(k_i+1)(k_i-\alpha)
   \E\big[f^{(\underline{k}-e_i)}(\underline{X}^\alpha_\infty)\big]
   \\
	&\;\;+(2-3\alpha)\big({\bf 1}_{k_1=k_2=0}+{\bf 1}_{k_2=k_3=0}+{\bf 1}_{k_3=k_1=0}\big)
  \\
	&\;\;\; +\frac{\alpha}{2}\sum_{i=1}^3{\bf 1}_{k_i=0}\sum_{j\neq i=1}^3
    \sum_{p_j=1}^{k_j}{k_j\choose p}\E\big[f^{(\underline{k}+(p-1)e_i-pe_j)}(\underline{X}^\alpha_\infty)\big]\Big),
\end{aligned}
\end{equation}
where $S=k_1+k_2+k_3$.
\label{Lem:002}
\end{lemma}


\begin{proof} Choose $\underline{k}\in\mathbb{N}_0^3\setminus\{(0,0,0),(1,0,0),(0,1,0),(0,0,1)\}$, and let $\alpha\in[0,1]$. As
for all $\underline{\eta}\in(0,1)^3$ and $i=1,2,3$,
\begin{equation}
\label{e:016a}
    \partial_i f^{(\underline{k})}(\underline{\eta})
    =\mathbf{1}_{\{k_i\not =0\}}k_i f^{(\underline{k}-e_i)}(\underline{\eta}),
\end{equation}
and all $i,j\in\{1,2,3\}$,
\begin{equation}
\label{e:016b}
    \partial_{i,j} f^{(\underline{k})}(\underline{\eta})
    =\mathbf{1}_{\{k_i,k_j\not =0\}}\big(k_i-\delta_{i,j}\big)k_j f^{(\underline{k}-e_i-e_j)}(\underline{\eta}),
\end{equation}
it follows that for all $\underline{\eta}\in(0,1)^3$,
\begin{equation}
\begin{aligned}
\label{e:017}
	\Omega_\alpha \Phi^{f^{\underline{k}}}(\chi)
  &=
     \int\mu^{\otimes 3}(\mathrm{d}\underline{u})\,\Big(\sum_{i,j=1}^{3}\mathbf{1}_{\{k_i,k_j\not =0\}}\big(k_i-\delta_{i,j}\big)k_j
     \big(\delta_{ij}f^{(\underline{k}-e_i)}(\underline{\eta})-f^{(\underline{k})}(\underline{\eta})\big)
     \\
     &\;\;
     +(2-\alpha)\sum_{i=1}^3\mathbf{1}_{\{k_i\not =0\}}k_i\big(f^{(\underline{k}-e_i)}(\underline{\eta})-3f^{(\underline{k})}(\underline{\eta})\big)
     \\
	&\;\; +(2-3\alpha)\sum_{1\le i<j\le 3}\mathbf{1}_{\{k_i=k_j=0\}}-3(2-3\alpha)f^{(\underline{k})}(\underline{\eta})
  \\
	&\;\; +\frac{\alpha}{2}\sum_{i=1}^3\mathbf{1}_{\{k_i=0\}}\sum_{j=1;j\not =i}^3{k_j\choose p_j }f^{(\underline{k}+(p_j-1)e_i-p_je_j)}(\underline{\eta})
\\
&\;\;-\alpha\sum_{i=1}^3\mathbf{1}_{\{k_i\not =0\}}\sum_{j=1;j\not =i}^3{k_j\choose p_j }f^{(\underline{k}-e_i)}(\underline{\eta}).
\end{aligned}
\end{equation}
Using that $\Omega_\alpha \Phi^{f^{\underline{k}}}(\chi_\infty)=0$ for all $\underline{k}$, implies that
\begin{equation}
\begin{aligned}
\label{e:017a}
	\E\big[f^{(\underline{k})}(\underline{X}^\alpha_\infty)\big]&\Big(\sum_{i,j=1}^{3}\mathbf{1}_{\{k_i,k_j\not =0\}}\big(k_i-\delta_{i,j}\big)k_j+3(2-\alpha)\sum_{i=1}^3\mathbf{1}_{\{k_i\not =0\}}k_i+3(2-3\alpha)\Big)
\\
  =&
  \Big(\sum_{i,j=1}^{3}\mathbf{1}_{\{k_i,k_j\not =0\}}\big(k_i-\delta_{i,j}\big)k_j\delta_{ij}+(2-\alpha)\sum_{i=1}^3\mathbf{1}_{\{k_i\not =0\}}k_i\Big)\E\big[f^{(\underline{k}-e_i)}(\underline{X}^\alpha_\infty)\big]\\
  &+(2-3\alpha)\sum_{1\le i<j\le 3}\mathbf{1}_{\{k_i=k_j=0\}}
  \\
  &+\frac{\alpha}{2}\sum_{i=1}^3\mathbf{1}_{\{k_i=0\}}\sum_{j=1;j\not =i}^3{k_j\choose p_j }\E\big[f^{(\underline{k}+(p_j-1)e_i-p_je_j)}(\underline{X}^\alpha_\infty)\big].
\end{aligned}
\end{equation}
Note that
\begin{equation}
\begin{aligned}
\label{e:017b}
   \sum_{i,j=1}^{3}&\mathbf{1}_{\{k_i,k_j\not =0\}}\big(k_i-\delta_{i,j}\big)k_j+3(2-\alpha)\sum_{i=1}^3\mathbf{1}_{\{k_i\not =0\}}k_i+3(2-3\alpha)
   \\
   &=
   \big(\sum_{i=1}^3k_i\big)^2-\sum_{i=1}^3k_i+3(2-\alpha)\sum_{i=1}^3k_i+3(2-3\alpha)
   \\
   &=
   \big(\sum_{i=1}^3k_i+3\big)\big(\sum_{i=1}^3k_i+(2-3\alpha)\big)
\end{aligned}
\end{equation}
and
\begin{equation}
\label{e:017c}
   \sum_{i,j=1}^{3}\mathbf{1}_{\{k_i,k_j\not =0\}}\big(k_i-\delta_{i,j}\big)k_j\delta_{ij}+(2-\alpha)\sum_{i=1}^3\mathbf{1}_{\{k_i\not =0\}}k_i
   -\alpha\sum_{i=1}^3\mathbf{1}_{\{k_i\not =0\}}
   =
   \sum_{i=1}^3\mathbf{1}_{\{k_i\not =0\}}(k_i+1)\big(k_i-\alpha\big),
\end{equation}
which finishes the proof for $\underline{k}\in\mathbb{N}_0^3\setminus\{(0,0,0),(1,0,0),(0,1,0),(0,0,1)\}$. It is trivial to show it for $\underline{k}\in\{(0,0,0),(1,0,0),(0,1,0),(0,0,1)\}$.
\end{proof}

For $k_2=k_3=0$, the recurrence relation becomes
\begin{align*}
	\E\left[f^{(k_1,0,0)}(\underline{X}^\alpha_\infty)\right]= & \frac{1}{(k_1+3)(k_1+2-3\alpha)}\bigg( (k_1+1)(k_1-\alpha)\E\left[f^{(k_1-e_1,0,0)}(\underline{X}^\alpha_\infty)\right]\\
	& +(2-3\alpha) +\alpha\sum_{p=1}^{k_1}{k_1\choose p}\E\left[f^{(k_1-p,p-1,0)}(\underline{X}^\alpha_\infty)\right]\bigg).
\end{align*}
We get in particular
\begin{align*}
	\E\left[\underline{X}^\alpha_{1,\infty}\right] & = \frac{1}{3},\\
	\E\left[(\underline{X}^\alpha_{1,\infty})^2\right] & = \frac{1}{5},\E\left[\underline{X}_{1,\infty}\underline{X}_{2,\infty}\right]  = \frac{1}{15}\\
	\E\left[(\underline{X}^\alpha_{1,\infty})^3\right] & 
=\frac{11-7\alpha}{15(5-3\alpha)},\\
	\E\left[(\underline{X}^\alpha_{1,\infty})^4\right] & = \frac{37-25\alpha}{63(5-3\alpha)},\\
	\E\left[(\underline{X}^\alpha_{1,\infty})^5\right] & = \frac{145-165\alpha+44\alpha^2}{42(5-3\alpha)(7-3\alpha)}.
\end{align*}

In the case $\alpha=0$, we have a general formula: for all $k_1\in\N$,
\begin{equation}
	\E\left[(\underline{X}^0_{1,\infty})^{k_1}\right]=\frac{2k_1(k_1(k_1+6)+11)+36}{3(k_1+1)(k_1+2)^2(k_1+3)}.
\end{equation}

\begin{proposition}[Representation in case of the Kingman algebraic measure tree]
Let $\PP_\mathrm{Kin}$ the law of the Kingman algebraic measure tree. Let $B_{1,2}$ and $B_{2,2}$ be two independent beta random variables, such that $B_{1,2}$ has law $\mathrm{Beta}(1,2)$ and $B_{2,2}$ has law $\mathrm{Beta}(2,2)$. Then for all $f\colon \Delta_2\rightarrow\R$ continuous bounded,
\begin{align*}
	\E_\mathrm{Kin}\left[\int_{T^3}\mu^{\otimes 3}(\mathrm{d}\underline{u}) f(\eta_{(T,c,\mu)}(\underline{u}))\right]=\frac{1}{6}\sum_{\pi\in\CS_3}\E\big[f\circ\pi^*(B_{1,2}B_{2,2},B_{1,2}(1-B_{2,2}),1-B_{1,2})\big],
\end{align*}
where $\CS_3$ is the set of permutations of $\{1,2,3\}$, and for $\pi\in\CS_3$, $\pi^*\colon\Delta_2\rightarrow\Delta_2$ is the induced map $\pi^*(\underline{x})=(x_{\pi(1)},x_{\pi(2)},x_{\pi(3)})$.
\label{P:001}
\end{proposition}

\begin{proof}
We write $\underline{\tilde{\eta}}:=(\tilde{\eta}_1,\tilde{\eta}_2,\tilde{\eta}_3):=(B_{1,2}B_{2,2},B_{1,2}(1-B_{2,2}),1-B_{1,2})$ and define the $\Delta_2$-valued random variable $\underline{\eta}=(\eta_1,\eta_2,\eta_3)$ by the relations, for all $f\colon \Delta_2\rightarrow\R$ continuous bounded,
\begin{equation}
	\E[f(\underline{\eta})]=\frac{1}{6}\sum_{\pi\in\CS_3}\E\big[f\circ\pi^*(\underline{\tilde{\eta}})\big]=\frac{1}{6}\sum_{\pi\in\CS_3}\E\big[f\circ\pi^*(B_{1,2}B_{2,2},B_{1,2}(1-B_{2,2}),1-B_{1,2})\big].
\end{equation}
To show the result, we prove that the moments of $\eta$ coincide with the moments of the subtree masses distribution in the equilibrium of the Kingman diffusion. For this, it is enough to show that the moments of $\eta$ satisfy the recurrence relations on moments obtained for the case $\alpha=0$, that is, for all $\underline{k}\in\N^3$,
\begin{equation}\label{e:momentrelation}
	\E\left[f^{(\underline{k})}(\underline{\eta})\right]= \frac{1}{(S+3)(S+2)}\left(\sum_{i=1}^3k_i(k_i+1)\E\left[f^{(\underline{k}-e_i)}(\underline{\eta})\right]+2\left({\bf 1}_{k_1=k_2=0}+{\bf 1}_{k_2=k_3=0}+{\bf 1}_{k_3=k_1=0}\right)\right),
\end{equation}
where $S=k_1+k_2+k_3$. We first calculate all the moments of $\underline{\tilde{\eta}}$. Since $B_{1,2}$ and $B_{2,2}$ are independent variables, we have
\begin{align*}
	\E\left[f^{(\underline{k})}(\underline{\tilde{\eta}})\right] & = \E\left[B_{1,2}^{k_1}B_{2,2}^{k_1}B_{1,2}^{k_2}(1-B_{2,2})^{k_2}(1-B_{1,2})^{k_3}\right]\\
	& = \E\left[B_{1,2}^{k_1+k_2}(1-B_{1,2})^{k_3}\right]\E\left[B_{2,2}^{k_1}(1-B_{2,2})^{k_2}\right]\\
	& = \frac{\Gamma(3)\Gamma(k_1+k_2+1)\Gamma(k_3+2)}{\Gamma(1)\Gamma(2)\Gamma(S+3)}\frac{\Gamma(4)\Gamma(k_1+2)\Gamma(k_2+2)}{\Gamma(2)^2\Gamma(k_1+k_2+4)}\\
	& = 12\frac{\prod_{k=j}^3\Gamma(k_j+2)}{\Gamma(S+3)}\frac{\Gamma(k_1+k_2+1)}{\Gamma(k_1+k_2+4)}.
\end{align*}
Thus,
\begin{align*}
	\E\left[f^{(\underline{k})}(\underline{\eta})\right] = 4\frac{\prod_{j=1}^3\Gamma(k_j+2)}{\Gamma(S+3)}\sum_{1\leq j<l\leq3}\frac{\Gamma(k_j+k_l+1)}{\Gamma(k_j+k_l+4)}.
\end{align*}

Suppose first that ${\bf 1}_{k_1=k_2=0}+{\bf 1}_{k_2=k_3=0}+{\bf 1}_{k_3=k_1=0}=0$, and recall that for all $n\in\N$, $\Gamma(n+1)=n\Gamma(n)=n!$. Then the right hand side of \eqref{e:momentrelation} is
\begin{align*}
	RHS &=\frac{1}{(S+2)(S+3)}\sum_{i=1}^3k_i(k_i+1)4\frac{\prod_{j=1}^3\Gamma(k_j-\delta_{ij}+2)}{\Gamma(S+2)}\sum_{1\leq j<l\leq3}\frac{\Gamma(k_j-\delta_{ij}+k_l-\delta_{il}+1)}{\Gamma(k_j-\delta_{ij}+k_l-\delta_{il}+4)}\\
	&=4\frac{\prod_{j=1}^3\Gamma(k_j+2)}{\Gamma(S+3)(S+3)}\sum_{\substack{\{j,l,m\}=\{1,2,3\}\\1\leq j<l\leq3}}\left(\frac{(k_j+k_l)\Gamma(k_j+k_l)}{\Gamma(k_j+k_l+3)}+k_i\frac{\Gamma(k_j+k_l+1)}{\Gamma(k_j+k_l+4)}\right)\\
	&=4\frac{\prod_{j=1}^3\Gamma(k_j+2)}{\Gamma(S+3)(S+3)}\sum_{\substack{\{j,l,m\}=\{1,2,3\}\\1\leq j<l\leq3}}\frac{\Gamma(k_j+k_l+1)}{\Gamma(k_j+k_l+4)}\left((k_j+k_l+3)+k_i\right)\\
	&=4\frac{\prod_{j=1}^3\Gamma(k_j+2)}{\Gamma(S+3)}\sum_{1\leq j<l\leq3}\frac{\Gamma(k_j+k_l+1)}{\Gamma(k_j+k_l+4)}.
\end{align*}
Suppose now that $k_2=k_3=0$. Then $f^{(\underline{k})}(\underline{\eta})=\eta_1^{k_1}$ and we have
\begin{equation}
	\E\left[\eta_1^{k_1}\right] =\frac{4}{k_1+2}\left(\frac{1}{6}+\frac{2}{(k_1+1)(k_1+2)(k_1+3)}\right),
\end{equation}
and we can easily check that for all $k_1\in\N$,
\begin{equation}
	\E\left[\eta_1^{k_1}\right]=\frac{1}{(k_1+2)(k_1+3)}\left(k_1(k_1+1)\E\left[\eta_1^{k_1-1}\right]+2\right).
\end{equation}
\end{proof}

\appendix
\section{The cladogram obtained from Kingman's coalescent}
\label{S:appendix}

We show here that the random cladogram obtained from the Kingman $m$-coalescent has the distribution of the $\alpha$-Ford model with $m$ leaves. We first define precisely what we call the random cladogram obtained from a finite coalescent.

Let $m\in\N$. We consider $(P_t)_{t\geq0}$ a realization of the Kingman $m$-coalescent, which is an element of $\CD_{\Pi^m}([0,\infty))$ the space of c\`adl\`ag paths on the set $\Pi^m$ of all partitions of $\{1,...,m\}$. Write $(\CP_k)_{k=m,...,1}$ for the sequence of states taken by $(P_t)_t$, such that $|\CP_k|=k$. For $k=1,..,m$, we denote by $C_{k,1},...,C_{k,k}$ the equivalence classes of $\CP_k$ such that the smallest element $c_{r,k}$ of $C_{k,r}$ satifies $c_{k,1}<...<c_{k,k}$.

We then define the $m$-cladogram
\begin{equation}
	\Fc=\Fc((P_t)_t)
\end{equation}
by the following construction (see Fig. \ref{f:cladocoal}):
\begin{itemize}
	\item $k=m$: for $r=1,...,m$, let $l_r$ be a vertex of $\Fc$. These $m$ vertices will be the leaves of $\Fc$, with label $r$ for $l_r$. To each $C_{k,r}$, we will associate a vertex $v_{k,r}$ in $\Fc$. For $r=1,...,m$, we define $v_{m,r}=l_r$.
	\item $k=m-1,...,2$: there exists a unique $i\in\{1,...,k\}$ such that $C_{k,i}$ is the union of $C_{k+1,i}$ and $C_{k+1,j}$ for some $j>i$. Let $v_k$ be a vertex of $\Fc$, and $\{v_k,v_{k+1,i}\}$ and $\{v_k,v_{k+1,i}\}$ two edges of $\Fc$. We define:
	\begin{equation}
		v_{k,r}:=
		\begin{cases}
		v_{k+1,r}&\text{ if }r\in\{1,...,i-1\}\cup\{i+1,...,j-1\},\\
		v_k&\text{ if }r=i,\\
		v_{k+1,r+1}&\text{ if }r\in\{j,...,k\}
		\end{cases}
	\end{equation}
	\item $k=1$: let $\{v_{2,1},v_{2,2}\}$ be an edge of $\Fc$.
\end{itemize}

We define this way an $m$-cladogram $\Fc((P_t)_t)$ associated to a realization $((P_t)_t)$ of the Kingman $m$-coalescent. Putting weight $1/m$ on each leaf, we also define an algebraic measure tree in $\T_2^m$.

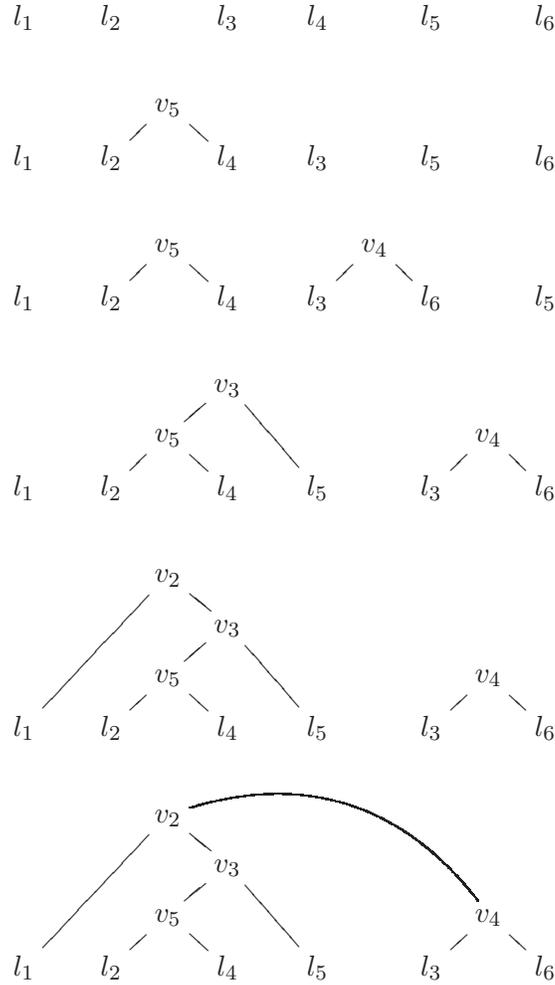
\begin{figure}[t!]
\[
\xymatrix@=0.5pc{
l_1&&l_2&&l_3&&l_4&&l_5&&l_6\\
&&&&&&&&&&\\
&&&v_5\ar@{-}[dl]\ar@{-}[dr]&&&&&&&\\
l_1&&l_2&&l_4&&l_3&&l_5&&l_6\\
&&&&&&&&&&\\
&&&v_5\ar@{-}[dl]\ar@{-}[dr]&&&&v_4\ar@{-}[dl]\ar@{-}[dr]&&&\\
l_1&&l_2&&l_4&&l_3&&l_6&&l_5\\
&&&&&&&&&&\\
&&&&v_3\ar@{-}[dl]\ar@{-}[ddrr]&&&&&&\\
&&&v_5\ar@{-}[dl]\ar@{-}[dr]&&&&&&v_4\ar@{-}[dl]\ar@{-}[dr]&\\
l_1&&l_2&&l_4&&l_5&&l_3&&l_6\\
&&&&&&&&&&\\
&&&v_2\ar@{-}[dr]\ar@{-}[dddlll]&&&&&&&\\
&&&&v_3\ar@{-}[dl]\ar@{-}[ddrr]&&&&&&\\
&&&v_5\ar@{-}[dl]\ar@{-}[dr]&&&&&&v_4\ar@{-}[dl]\ar@{-}[dr]&\\
l_1&&l_2&&l_4&&l_5&&l_3&&l_6\\
&&&&&&&&&&\\
&&&v_2\ar@{-}[dr]\ar@{-}[dddlll]\ar@/^2pc/@{-}[ddrrrrrr]&&&&&&&\\
&&&&v_3\ar@{-}[dl]\ar@{-}[ddrr]&&&&&&\\
&&&v_5\ar@{-}[dl]\ar@{-}[dr]&&&&&&v_4\ar@{-}[dl]\ar@{-}[dr]&\\
l_1&&l_2&&l_4&&l_5&&l_3&&l_6\\
}\]
\caption{Construction of the $6$-cladogram associated with the realization of a $6$-coalescent where $2$ and $4$ coalesce first, then $3$ and $6$, then $\{2,4\}$ and $\{3,6\}$, and so on...}\label{f:cladocoal}
\end{figure}

From this construction, we have that, for each $N\in\N$, the Kingman $N$-coalescent defines a random algebraic measure tree in $\T_2^N$, that we denote by $\tau_N$. The sequence $(\tau_N)_N$ takes values in the compact space $\T_2$, so we can show that the sequence is convergent by proving that each convergent subsequence converges to the same limit. But this results from the consistency property of the Kingman coalescents: for $N$ and $N'$ in $\N$, and for all $m$, the shape spanned by $m$-points sampled from $\tau_N$ and $\tau_{N'}$ have the same distribution in $\FC_m$. Therefore, we can introduced the Kingman algebraic measure tree in the following way.

\begin{proposition}
Let $m\in\N$ and $(K^m_t)_{t\geq0}$ be the Kingman $m$-coalescent. Then the random cladogram
\begin{equation}
	\Fc((K^m_t)_{t\geq0})
\end{equation}
has the distribution of the $(\alpha=0)$-Ford model on $m$-cladograms.
\end{proposition}

\begin{proof}
We show the result by induction. For $m=1,2,3$, the result is trivial since there is only one $m$-cladogram. Suppose $m\geq4$ and that the result is true for $m-1$.

For both random cladograms, we have exchangeability, that is both distributions are symmetric under permutation of leaf labels. For this reason, we can focus on the shape of the cladograms, i.e.\ the cladograms where we forgot the labelling. To show our claim, we use that the two construction procedures described above are somehow reverse and that the first coalescence event in the Kingman coalescnt corresponds to the last insertion of a leaf in the $(\alpha=0)$-Ford model.

Let $(K^m_t)_t$ be the Kingman $m$-coalescent. We consider the construction of $\Fc=\Fc((K^m_t)_t)$ and ignore the leaf labels. We now see $C_{k,r}$ and $v_{k,r}$, $k=1,...,m$, $r=1,...,k$ as random objects. The second step of the construction ($k=m-1$) consists in defining a vertex $v_{m-1}$ of $\Fc$, and the two edges $\{v_{m-1},v_{m,i}\}$ and $\{v_{m-1},v_{m,i}\}$ of $\Fc$, according to the first coalescence in $K^m$.

Define the processus $(\widetilde{K}_t^{m-1})_{t\geq0}$ on $\Pi^{m-1}$ by declaring that $a,b\in\{1,...,m-1\}$ are in the same block of the partition $\widetilde{K}_t^{m-1}$ if and only if $C_{m,a}$ and $C_{m,b}$ lie in the same block of $K_{\tau+t}^m$ where $\tau$ is the time of the first coalescence. Then we know that $(\widetilde{K}_t^{m-1})_{t\geq0}$ is the Kingman $(m-1)$-coalescent. Therefore, by the induction hypothesis, the cladogram obtained in the construction of $\Fc$ from the vertices $v_{m-1,1},...,v_{m-1,m-1}$ has the distribution of the $(\alpha=0)$-Ford model on $(m-1)$-cladograms.

Now we know that the first coalescence in $K^m$ is independent from $(\widetilde{K}_t^{m-1})_{t\geq0}$. Therefore the second step of the construction consists in choosing uniformly at random one leaf of $(\Fc(\widetilde{K}_t^{m-1}))_{t\geq0}$ and in adding two adjacent edges with vertices to this leaf. This step corresponds exactly to the $m$-th step of the construction in the $(\alpha=0)$-Ford model, where we forgot the leaf labels.
\end{proof}

Using this result and the consistency property of the Kingman finite coalescents, we get what Ford calls the \emph{deletion stability} for the $(\alpha=0)$-Ford model. That is, consider the $m$-cladogram with law the $(\alpha=0)$-Ford model. If we remove at random a leaf and relabel the leaves from $1$ to $m-1$ in a way that the order remains the same, we obtain a random $(m-1)$-cladogram with distribution the $(\alpha=0)$-Ford model. See \cite[Proposition 42]{Ford2005} for a proof of the deletion stability of the $\alpha$-Ford model for all $\alpha\in[0,1]$ through a combinatorial argument.

\bibliography{kingman}
\bibliographystyle{alpha}
 \end{document}